\documentclass[a4paper,11pt,reqno]{amsart}
\usepackage{amsmath}
\usepackage{amsthm,enumerate}

\usepackage{graphicx}
\usepackage{amssymb}

\usepackage{appendix}

\usepackage[italian,english]{babel}

\usepackage[utf8x]{inputenc}
\usepackage{fancyhdr}

\usepackage{calc}
\usepackage{url}

\usepackage[text={6in,8.6in},centering]{geometry}

\usepackage{srcltx, inputenc}

\usepackage[T1]{fontenc}

\usepackage[usenames,dvipsnames]{color}

\makeatletter
\def\cleardoublepage{\clearpage\if@twoside \ifodd\c@page\else%
         \hbox{}%
     \thispagestyle{empty}
     \newpage%
     \if@twocolumn\hbox{}\newpage\fi\fi\fi}
\makeatother

\let\cleardoublepage\clearpage

\addto\captionsitalian{}

\newtheorem{thm}{Theorem}[section]

\newtheorem{lem}[thm]{Lemma}
\newtheorem{pro}[thm]{Proposition}
\newtheorem{den}[thm]{Definition}
\newtheorem{oss}[thm]{Remark}
\numberwithin{equation}{section}

\newcommand{\m}{\mathsf{m}}

\begin{document}

\title[]{Blow-up and global existence \\ for the porous medium equation with reaction \\ on a class of Cartan-Hadamard manifolds}

\author {Gabriele Grillo} \author[Matteo Muratori]{Matteo Muratori}
\author{Fabio Punzo}

\address {Gabriele Grillo, Matteo Muratori, Fabio Punzo: Dipartimento di Matematica, Politecnico di Milano, Via Bonardi 9, 20133 Milano, Italy}
\email{gabriele.grillo@polimi.it}
\email{matteo.muratori@polimi.it}
\email{fabio.punzo@polimi.it}


%
%
%
%


\maketitle

\scriptsize
\noindent{\bf Abstract}. We consider the porous medium equation with power-type reaction terms $u^p$ on negatively curved Riemannian manifolds, and solutions corresponding to bounded, nonnegative and compactly supported data. If $p>m$, small data give rise to global-in-time solutions while solutions associated to large data blow up in finite time. If $p<m$, large data blow up at worst in infinite time, and under the stronger restriction $p\in(1,(1+m)/2]$ all data give rise to solutions existing globally in time, whereas solutions corresponding to large data blow up in infinite time. The results are in several aspects significantly different from the Euclidean ones, as has to be expected since negative curvature is known to give rise to faster diffusion properties of the porous medium equation.

\normalsize


\section{Introduction}
We consider solutions to the nonlinear evolution problem
\begin{equation}\label{eq1}
\begin{cases}
u_t \,=\, \Delta (u^m)+ u^p & \textrm{in}\;\; M\times (0,T)\,, \\
u \,=\, u_0 & \textrm{in}\;\; M\times \{0\}\,,
\end{cases}
\end{equation}
where $M$ is an $N$-dimensional complete, simply connected Riemannian manifold with nonpositive sectional curvatures (namely a \emph{Cartan-Hadamard manifold}) and $\Delta$ is the Laplace-Beltrami
operator on $M$, $m>1, p>1, T\in (0, \infty]$. The assumption $m>1$ corresponds to consider the \it slow diffusion \rm case, see \cite{Vaz07}. The initial datum $u_0$ will be always assumed to be nonnegative and bounded. We shall always suppose, without further comment, that $N\ge2$.

We shall concentrate on the situation in which the curvature bounds
\[
\textrm{Ric}_o(x)\leq - (N-1)  h^2 \qquad\textrm{or}\qquad \textrm{Ric}_o(x)\geq - (N-1)  k^2
\]
hold for some $h,k>0$, where $\textrm{Ric}_o(x)$ is the
\emph{Ricci curvature} at $x$ in the radial direction
$\frac{\partial}{\partial r}$. In particular a global negative curvature condition like the one  above,  involving sectional curvatures,  implies that the spectrum of $-\Delta$ on $M$ is bounded away from zero, see \cite{M}, this being in sharp contrast with the Euclidean setting and resembling, to some extent, the case of a \it bounded \rm Euclidean domain. In fact, the basic example we have in mind is the \it hyperbolic space \rm $\mathbb{H}^n_h$, namely the complete, simply connected manifold with sectional curvatures everywhere equal to $-h^2$. It is known that, on $\mathbb{H}^n_h$, Brownian motion, associated to $-\Delta$ by a standard procedure, has an expected speed which is \it linear \rm for large times (see e.g. \cite[Cor. 5.7.3]{D}), hence diffusion occurs at a much \it faster \rm rate than the one typical of the Euclidean situation.

The behaviour of solutions to the \it porous medium \rm or \it fast diffusion \rm equations
\begin{equation}\label{pme}
u_t \,=\, \Delta (u^m) \quad \textrm{in}\;\; M\times (0,T)\,.
\end{equation}
has been the subject of recent results (see e.g. \cite{GMhyp}, \cite{GM2}, \cite{GMPbd}, \cite{GMPrm}, \cite{GMV}, \cite{Pu1}, \cite{Pu2}, \cite{PuAA},  \cite{VazH}). That even \it nonlinear \rm diffusion gives rise to speedup phenomena can be seen at least in two different ways, see \cite{VazH} for $M=\mathbb H^N$, and \cite{GMV} for extensions on some more general manifolds. First, the $L^\infty$ norm of a solution corresponding to a compactly supported datum obeys the law $\|u(t)\|_\infty\asymp \left(\frac{\log t}t\right)^{1/(m-1)}$ as $t\to+\infty$, which is \it faster \rm than the corresponding Euclidean bound, the latter being given in term of a power $t^{-\alpha(N,m)}$ with $\alpha(N,m)=(m-1+(2/N))^{-1}<1/(m-1)$. Besides, the volume $\mathsf{V}(t)$ of the support of the solution $u(t)$ satisfies $\mathsf{V}(t)\asymp t^{1/(m-1)}$ while in the Euclidean situation one has $\mathsf{V}(t)\asymp t^{\beta(N,m)}$ with $\beta(N,m)=n(2+n(m-1))^{-1}=<1/(m-1)$.

The behaviour of solutions to problem \eqref{eq1} is therefore determined by competing phenomena: the diffusion pattern associated to $-\Delta$, the reaction due to the power source, and the (slow, but faster than in the Euclidean case) diffusion properties of the porous medium equation $u_t=\Delta(u^m)$. In fact, in the case of linear diffusion ($m=1$) it is known (see \cite{BPT}, \cite{WY}, \cite{WY2}, \cite{Pu3}) that, when $M=\mathbb{H}^N$, for all $p>1$ and sufficiently small nonnegative data there exists a global in time solution. The situation is different in $\mathbb R^N$: indeed, blowup occurs for all nontrivial nonnegative data when $p\le1+2/N$, while global existence prevails for $p>1+ 2/N$ (for more specific results, see e.g. \cite{CFG}, \cite{DL}, \cite{F}, \cite{FI}, \cite{H}, \cite{L}, \cite{Q}, \cite{S}, \cite{W}, \cite{Y}.)

To understand more precisely the differences between the Euclidean results and the ones proved in the present paper, let us first summarize qualitatively some of the former ones, quoting from \cite{SGKM}. For subsequent, more detailed results see e.g. \cite{GV}, \cite{MQV}, \cite{Vaz1} and references quoted therein.

\smallskip

\noindent \bf The case $M=\mathbb{R}^N$\rm. We suppose that the initial datum is \it nonnegative, nontrivial and compactly supported\rm. In this case we have:
\begin{itemize}
\item (\cite[Th. 1, p. 216]{SGKM}) For any $p>1$, all sufficiently large data give rise to solutions blowing up in finite time;
\item (\cite[Th. 2, p. 217]{SGKM}) if $p\in\left(1,m+\frac2N\right)$, \it all \rm data give rise to solutions blowing up in finite time;
\item (\cite[Th. 3, p. 220]{SGKM}) if $p>m+\frac2N$, all sufficiently small data give rise to solutions existing globally in time.

\end{itemize}
Let us mention that further nonexistence results for quasilinear parabolic equations, also involving $p$-Laplace type operators, have been obtained in \cite{MP}, \cite{MP2}, \cite{PT} (see also \cite{MMP} for the case of Riemannian manifolds). Moreover, in \cite{Sacks} problem
\[\begin{cases}
u_t \,=\, \Delta (u^m)+ \lambda u^p & \textrm{in}\;\; \Omega\times (0,T)\,, \\
u \,=\, 0 & \textrm{on}\;\; \partial \Omega\times (0,T)\,,\\
u \,=\, u_0 & \textrm{in}\;\; \Omega\times \{0\}\,,
\end{cases}
\]
where $\Omega$ is a bounded domain of $\mathbb R^N$ and $\lambda$ is a positive parameter, has been studied. Let $\lambda_1(\Omega)$ be the first eigenvalue of $-\Delta$ in $\Omega$, completed with homogeneous Dirichlet boundary conditions. It is shown that (see \cite[Theorem 1.3]{Sacks}) there exists a global solution for any $u_0\in L^q(\Omega), q>1$, and for any $\lambda\leq \lambda_1(\Omega).$ In addition, when $p>m$, or $p=m$ and $\lambda>\lambda_1(\Omega)$ (see \cite[Section 4]{Sacks}), then depending on the initial datum $u_0$, solutions may or may not exist for all times. Analogous results can also be found in \cite[Chapter VII]{SGKM}.

Existence of global solutions and blow-up in finite time for problem \eqref{eq1} have been studied in \cite{Z}, if the volume of geodesic balls grows as a power of the radius $R$, namely as $R^{\alpha}$ with $\alpha\geq 2$; this occurs, in particular, when Ricci curvature is nonnegative. However, we should note that such a condition tipically is not satisfied in our setting, in which the volume of geodesic balls can grow exponentially or faster with respect to the radius.
In particular, in \cite{Z} it is proved that if $m<p\leq m+2/{\alpha}$, then problem \eqref{eq1} does not have global (nontrivial) solutions. Instead, if $\alpha=N, p>m+2/N$, under a suitable assumption on the metric tensor, there exists a global solution of \eqref{eq1}, for some $u_0.$ Such results extend some of those in \cite{SGKM} to general Riemannian manifolds.

The situation on negatively curved manifolds is significantly different, as we now briefly summarize by singling out qualitatively some of our results.

\smallskip

\noindent \bf The case of a Cartan-Hadamard manifold $M$\rm. We suppose that the initial datum is \it nonnegative, nontrivial and compactly supported\rm. In this case we have:
\begin{itemize}
\item (see Theorems \ref{super p>m}, \ref{eigen p>m}) If $p>m$ and upper curvature bounds (see \eqref{H3}) hold, all sufficiently small data give rise to solutions existing globally in time. Besides, a class of sufficiently small data shows propagation properties \it identical \rm to the ones valid for the unforced porous medium equation \eqref{pme}. Moreover, small data non necessarily with compact support and possibly with arbitrarily large $L^p$ norms ($p\ge1$) give rise to solutions existing globally in time if $p>m$, and also if $p=m$ and a suitable curvature bound holds;
\item (see Theorem \ref{sub p>m}) If $p>m$ and lower curvature bounds (see \eqref{H2}) hold, all sufficiently large data give rise to solutions blowing up in finite time;
\item (see Theorem \ref{super p<m}). If $p\in\left(1,\frac{1+m}{2}\right]$ and upper curvature bounds (see \eqref{H3}) hold, all data exist globally in time;
\item (see Theorem \ref{sub p<m}). If $p\in(1,m)$ and lower curvature bounds (see \eqref{H2}) hold, all sufficiently large data give rise to solutions blowing up at worst in infinite time.

\end{itemize}

Thus the overall picture is considerably different from the Euclidean one, on the one hand since the main critical exponent turns out to be $p=m$, on the other hand since a completely new phenomenon, namely existence of solutions blowing up in infinite time, appears when $p\le(1+m)/2$. We do not know, and leave these as  challenging open problems, whether for $p\in\left(\frac{1+m}{2},m\right)$ solutions corresponding to small initial data exist for all time and if solutions corresponding to large data blow up in finite or infinite time.

The paper is organized as follow. Section 2 contains some geometric preliminaries, the relevant notation, a concise discussion of Laplacian comparison in Riemannian geometry, and, finally, a brief discussion of local existence of solution to \eqref{eq1} and comparison principles. Section 3 contains the statements of our main results. Section 4 contains two general auxiliary lemmas, that will be repeatedly used in the construction of the barriers we need in the proofs of our main results. Such proofs are contained in Section 5.

\section{Preliminaries}
\subsection{Notations from Riemannian geometry}
Let $M$ be a complete noncompact Riemannian manifold. Let $\Delta$
denote the Laplace-Beltrami operator, $\nabla$ the Riemannian
gradient and $d\mathcal{V}$ the Riemannian volume element on $M$.


We consider \emph{Cartan-Hadamard} manifolds, i.e.~complete, noncompact, simply connected Riemannian manifolds with nonpositive
sectional curvatures everywhere. Observe that on Cartan-Hadamard manifolds the \emph{cut locus} of any point $o$ is empty \cite{Grig,Grig3}.
Hence, for any $x\in M\setminus \{o\}$ one can define its {\it polar coordinates} with pole at $o$, namely $r(x) := d(x, o)$ and $\theta\in \mathbb S^{N-1}$. If we denote by $B_R$ the Riemannian ball of radius $R$ centred at $o$ and $ S_R:=\partial B_R $, there holds
\[
\textrm{meas}(S_R)\,=\, \int_{\mathbb S^{N-1}}A(R, \theta) \, d\theta^1d \theta^2 \ldots d\theta^{N-1}\,,
\]
for a specific positive function $A$ which is related to the metric tensor, \cite[Sect. 3]{Grig}. Moreover, it is direct to see that the Laplace-Beltrami operator in polar coordinates has the form
\begin{equation}\label{eq3}
\Delta \,=\, \frac{\partial^2}{\partial r^2} + \m (r, \theta) \, \frac{\partial}{\partial r} + \Delta_{S_{r}} \, ,
\end{equation}
where $\m(r, \theta):=\frac{\partial }{\partial r}(\log A)$ and $ \Delta_{S_{r}} $ is the Laplace-Beltrami operator on $ S_{r} $.  Thanks to \eqref{eq3}, we have that
$$ \m(r,\theta) =\Delta r(x)\quad \textrm{for every}\,\, x\equiv (r, \theta)\in M\setminus\{o\}\,.$$

Let $$\mathcal A:=\left\{f\in C^\infty((0,\infty))\cap C^1([0,\infty)): \, f'(0)=1, \, f(0)=0, \, f>0 \ \textrm{in}\;\, (0,\infty)\right\} .$$ We say that $M$ is a {\it spherically symmetric manifold} or a {\it model manifold} if the Riemannian metric is given by
\[
ds^2 \,=\, d r^2+\psi(r)^2 \, d\theta^2,
\]
where $d\theta^2$ is the standard metric on $\mathbb S^{N-1}$ and $\psi\in \mathcal A$. In this case, we shall write $M\equiv M_\psi$; furthermore, we have $A(r,\theta)=\psi(r)^{N-1} \, \eta(\theta) $ for a suitable angular function $\eta$, so that
\[
\Delta \,=\, \frac{\partial^2}{\partial r^2} + (N-1) \, \frac{\psi'}{\psi} \, \frac{\partial}{\partial r} + \frac1{\psi^2} \, \Delta_{\mathbb S^{N-1}} \, .
\]
Note that $\psi(r)=r$ corresponds to $M=\mathbb R^N$, while $\psi(r)=\sinh r$ corresponds to $ M=\mathbb H^N $, namely the $N$-dimensional hyperbolic space.

For any $x\in M\setminus\{o\}$, we denote by $\textrm{Ric}_o(x)$ the
\emph{Ricci curvature} at $x$ in the radial direction
$\frac{\partial}{\partial r}$. Let $\omega$ be any pair of tangent
vectors from $T_x M$ having the form $\big(\frac{\partial}{\partial r}
, V \big)$, where $V$ is a unit vector orthogonal to
$\frac{\partial}{\partial r}$. We denote by $\textrm{K}_{\omega}(x)$ the \emph{sectional curvature} at $x\in M$ of the $2$-section determined by $\omega$.

\subsection{Laplacian comparison}

Let us recall some crucial Laplacian comparison results. It is by now classical (see e.g.~\cite{GW} and \cite[Section
15]{Grig}) that if
\[
\textrm{Ric}_{o}(x) \geq -(N-1) \, k^2 \quad \textrm{for all } x\equiv(r,\theta)\in M\setminus\{o\}
\]
for some $k>0$, then
\[
\m(r, \theta) \leq (N-1) k \coth(kr) \quad \textrm{for all}\;\; r>0, \theta \in \mathbb S^{N-1}\,.
\]
So, in particular,
\begin{equation}\label{eq12}
\m(r, \theta) \leq (N-1) k \coth(k) \quad \textrm{for all}\;\; r\geq 1, \theta \in \mathbb S^{N-1}\,,
\end{equation}
since $r\mapsto \coth r$ is decreasing. On the other hand, if
\begin{equation}\label{eq6}
\textrm{Ric}_{o}(x)\leq - (N-1)h^2\quad \textrm{for all } x \equiv (r,\theta)\in M\setminus\{o\}
\end{equation}
for some $h>0$, then (see \cite [Theorem 2.15]{Xin})
\begin{equation}\label{eq13}
\m(r, \theta) \geq (N-1) h\,\coth(hr)\ge (N-1) h  \quad \textrm{for all}\;\; r >0, \theta \in \mathbb S^{N-1}\,
\end{equation}
(the second inequality is merely due to the fact that $\coth(r)\geq 1$ for all $r>0$). Let us observe that the latter implication is based upon the assumption that $M$ is a Cartan-Hadamard manifold.
Indeed, on general Riemannian manifolds with a pole $o$, namely with a point $o\in M$ having empty cut locus, inequality \eqref{eq13} is valid, provided that
\[\textrm{K}_{\omega}(x)\leq -h^2\quad \textrm{for all } x \equiv (r,\theta)\in M\setminus\{o\}\,.
\]
Clearly, \eqref{eq6} is a weaker condition than the previous one concerning the sectional curvature.

In the special case of a model manifold $M_\psi$, for any $x\equiv(r, \theta)\in M_\psi\setminus\{o\}$ we have
\[
\textrm{K}_{\omega}(x)=-\frac{\psi''(r)}{\psi(r)} \, , \quad \textrm{Ric}_{o}(x)=-(N-1) \, \frac{\psi''(r)}{\psi(r)} \, .
\]
In particular, as $\psi\in \mathcal A$, the condition $ \psi''\geq 0 $ in $(0, \infty)$ is necessary and sufficient for $ M_\psi $ to be a Cartan-Hadamard manifold. Finally, note that for any Cartan-Hadamard manifold we have $\textrm{K}_{\omega}(x)\leq 0$, therefore the Laplace comparison theorem easily gives that
\[
\m(r, \theta) \geq \frac{N-1}{r} \quad \textrm{for any } x \equiv (r, \theta) \in M \setminus \{o\}\,.
\]

\subsection{Main assumptions and consequences}
Throughout the paper we shall work under the following assumption:

\begin{equation} \label{H1}
M \ \textrm{is a Cartan-Hadamard manifold of dimension $N\ge2$}.
\end{equation}

\medskip Besides, one or both the following curvature bounds will be required:
\begin{equation} \label{H2}
\textrm{Ric}_o(x)\geq - (N-1)  k^2 \ \textrm{for some } k>0\, ;
\end{equation}
\begin{equation} \label{H3}
\textrm{Ric}_o(x) \leq -  (N-1)h^2 \ \textrm{for some } h>0\, .
\end{equation}

\subsection{Local existence and comparison}\label{existence}

In this brief section we first give a precise meaning to the concept of solution to \eqref{eq1} we shall deal with, and then establish some elementary existence results and comparison principles, which are essential to be able to exploit all of the barrier functions we provide below.

\begin{den}\label{def-sol-vw}
Let $ u_0 \in L^\infty(M) $, with $ u_0 \ge 0 $. Let $T>0$ and $ p,m>1 $. We say that a nonnegative function $ u \in L^\infty(M \times (0,S)) $ (for all $ S<T $) is a (very weak) solution to problem \eqref{eq1} if it satisfies
\begin{equation}\label{eq-vw}
-\int_M \int_0^T u \, \varphi_t \, dt \, d\mathcal{V} = \int_M u_0(x) \, \varphi(x,0) \, d\mathcal{V}(x) + \int_M \int_0^T \left( u^m \, \Delta \varphi + u^p \, \varphi \right) dt \, d\mathcal{V}
\end{equation}
for all nonnegative $ \varphi \in C^\infty_c(M\times[0,T)) $.

Similarly, we say that a nonnegative function $ u \in L^\infty(M \times (0,S)) $ (for all $ S<T $) is a (very weak) subsolution [supersolution] to problem \eqref{eq1} if it satisfies \eqref{eq-vw} with ``$=$'' replaced by ``$\le$'' [``$ \ge $''].
\end{den}

In order to give a (rather standard) local existence result, let us briefly discuss about \emph{minimal solutions}. To this end, we first need to introduce the auxiliary problems (let $ R>0 $)
\begin{equation}\label{approx-BR}
\begin{cases}
u_t = \Delta\!\left( u^m \right) + u^p & \text{in } B_R \times (0,T) \, , \\
u = 0 & \text{on } \partial B_R \times (0,T) \, , \\
u(\cdot,0) = u_0 & \text{in } B_R \, .
\end{cases}
\end{equation}
\begin{den}
Let $ u_0 \in L^\infty(B_R) $, with $ u_0 \ge 0 $. Let $T>0$ and $ p,m>1 $. We say that a nonnegative function $ u \in L^\infty(B_R \times (0,S)) $ (for all $ S<T $) is a (very weak) solution to problem \eqref{approx-BR} if it satisfies
\begin{equation}\label{eq-vwR}
-\int_{B_R} \int_0^T u \, \varphi_t \, dt \, d\mathcal{V} = \int_{B_R} u_0(x) \, \varphi(x,0) \, d\mathcal{V}(x) + \int_{B_R} \int_0^T \left( u^m \, \Delta \varphi + u^p \, \varphi \right) dt \, d\mathcal{V}
\end{equation}
for all nonnegative $ \varphi \in C^\infty(\overline{B}_R\times[0,T])$ with $\varphi=0$ on $\partial B_R$ for all $t\in[0,T]$ and at $t=T$.

Similarly, we say that a nonnegative function $ u \in L^\infty(B_R \times (0,S)) $ (for all $ S<T $) is a (very weak) subsolution [supersolution] to problem \eqref{approx-BR} if it satisfies \eqref{eq-vwR} with ``$=$'' replaced by ``$\le$'' [``$ \ge $''].
\end{den}
Note that problem \eqref{approx-BR} admits a nonnegative solution $ u_R \in L^\infty(B_R \times (0,S))$, for all $S< T_R$, where $T_R$ is the maximal existence time, i.e. $\|u_R(t)\|_{\infty}\to +\infty$ as $t\to T_R^-$; moreover, for problem \eqref{approx-BR} the comparison principle between sub-- and supersolutions holds (see \cite{ACP}). Observe that $ T_R\geq T$ for any $R>0$, where $T$ can be quantified depending on the initial datum $ u_0 $ by simple comparison with the solution of the associated ODE:
\[
\begin{cases}
x^\prime = x^p \, , \\
x(0)=\| u_0 \|_\infty \, ,
\end{cases}
\]
that is
\begin{equation}\label{approx-BR-1}
u_R(x,t) \le \frac{\| u_0 \|_\infty}{\left( 1-(p-1) \| u_0 \|_\infty^{p-1} \, t \right)^{\frac{1}{p-1}}} \qquad \Longrightarrow \qquad T_R \ge T:= \frac{1}{(p-1) \| u_0 \|_\infty^{p-1}} \, .
\end{equation}
Moreover, such a solution is unique. In particular, $ u_{R} \le u_{R+1} $, so that the family $ u_R $ is monotone increasing and, thanks to the upper bound \eqref{approx-BR-1}, it converges as $ R \to \infty $ to some solution $ u $ to \eqref{eq1}. Such a solution is necessary smaller than any other solution, due again to comparison on balls (see also Proposition \ref{thm:cp1} below). In this sense it is referred to as \emph{minimal}. We can define the \emph{maximal existence time} $T$ of $ u $ as the supremum over all $ S > 0 $ for which $\lim_{R\to\infty} u_R \in L^\infty(M \times (0,S)) $: note that $u$ does solve \eqref{eq1}, at least in the sense of Definition \ref{def-sol-vw}, up to such time.

\smallskip

As a reference for these results, we quote e.g.~\cite{ACP}, where in fact the authors mainly discuss the one-dimensional Euclidean case,  for more general nonlinearities. However, their arguments are easily adaptable to our framework as well. We omit details. See also \cite{PZ1,PZ2} for similar techniques applied to a related (but substantially different) problem in general Euclidean space.

\smallskip

In agreement with the above discussion, we can state the following existence result.

\begin{pro}[Existence]\label{thm:ex}
Let $ u_0 \in L^\infty(M) $, with $ u_0 \ge 0 $. Then there exists a solution to problem \eqref{eq1}, in the sense of Definition \ref{def-sol-vw}, with $$ T=\frac{1}{(p-1) \| u_0 \|_\infty^{p-1}} \, , $$
which is obtained as a monotone limit of the solutions to the approximate problems \eqref{approx-BR}. Moreover, such a solution is \emph{minimal}, in the sense that any other solution is larger.
\end{pro}

By taking advantage of the construction of the minimal solution, we can readily prove a fundamental comparison theorem.

\begin{pro}[Comparison with supersolutions]\label{thm:cp1}
Let $ u_0 \in L^\infty(M) $, with $ u_0 \ge 0 $. Let $ \overline{u} $ be a supersolution to \eqref{eq1} (for some $T>0$), according to Definition \ref{def-sol-vw}. Then, if $u$ is the minimal solution provided by Proposition \ref{thm:ex}, there holds
\begin{equation}\label{eq-cp1}
u \le \overline{u} \qquad \text{in } M \times (0,T) \, .
\end{equation}
In particular, if the supersolution exists at least up to time $ T $, then also $u$ does, i.e.~the maximal existence time for $ u $ is at least $T$.
\end{pro}
\begin{proof}
It is enough to apply the above-mentioned comparison results in $ B_R $: since $ \overline{u} $ is clearly also a supersolution to \eqref{approx-BR}, for each $ R>0 $, we have
\begin{equation}\label{eq-cp1-R}
u_R \le \overline{u} \qquad \text{in } B_R \times (0,T) \, .
\end{equation}
By passing to the limit as $ R \to \infty $ in \eqref{eq-cp1-R} we obtain \eqref{eq-cp1}, which trivially ensures that $ u $ does exist at least up to $T$, by the definition of maximal existence time.
\end{proof}

We also have a similar result for subsolutions.

\begin{pro}[Comparison with subsolutions]\label{thm:cp2}
Let $ u_0 \in L^\infty(M) $, with $ u_0 \ge 0 $. Let $ u $ be a solution (for some $T \equiv T_1 > 0$) and $ \underline{u} $ be a subsolution (for some $T \equiv T_2 > 0$) to \eqref{eq1}, according to Definition \ref{def-sol-vw}. Suppose that $ \underline{u} $ has the following additional property:
\[
\operatorname{supp} \underline{u} \vert_{M \times [0,S]} \ \ \text{is \emph{compact} for all } S < T_2 \, .
\]
Then
\begin{equation}\label{eq-cp2}
u \ge \underline{u} \qquad \text{in } M \times (0, T_1 \wedge T_2 ) \, .
\end{equation}
\end{pro}
\begin{proof}
Again, it is sufficient to apply comparison on balls: we fix any $ S < T_1 \wedge T_2 $ and observe that, if $ R $ is so large that $ \operatorname{supp} \underline{u} \rfloor_{M \times [0,S]} \subset B_R $, then $ u $ and $ \underline{u} $ are a supersolution and a subsolution, respectively, to \eqref{approx-BR}, whence
$$
u \ge \underline{u} \qquad \text{in } B_R \times (0,S) \, .
$$
Inequality \eqref{eq-cp2} then just follows by letting $ R \to \infty $ and using the arbitrariness of $S$.
\end{proof}

In the sequel, by ``solution'' to \eqref{eq1} we shall tacitly mean the minimal one, according to Proposition \ref{thm:ex}, to which therefore the crucial comparison results provided by Propositions \ref{thm:cp1}--\ref{thm:cp2} are directly applicable.

\medskip

\begin{oss} \rm It can be shown that if, for some $C>0,$
\[\operatorname{Ric}_o(x) \geq - C(1+ r(x))^2 \quad \textrm{for all}\;\; x\in M\,, \]
then the comparison principle between {\it any} bounded sub-- and supersolution holds for problem \eqref{eq1}; consequently, problem \eqref{eq1} admits a unique solution in $L^\infty(M\times (0, T))$. These results follow by combining the arguments of \cite{GMPbd} (where the Cauchy problem for \eqref{pme} has been dealt with) and those in \cite{ACP}, in order to consider the source term $u^p$. In this direction, let us mention that such a hypothesis on the Ricci curvature to get uniqueness is quite natural; see, e.g., \cite{I,I2,IM} for the linear case, $m=1$, without source terms. We omit the details, since a general comparison principle for problem \eqref{eq1} is not the main concern of the present paper; in fact, in order to prove our results, we do not need it, but  it is sufficient to use the comparison principle in the form of Propositions \ref{thm:cp1}, \ref{thm:cp2}.
\end{oss}


\section{Statements of the main results}

Our first main result concerns global existence of solutions for sufficiently small, compactly supported data, in the case $p>m$. Besides, such small data show propagation properties \it identical \rm to the ones valid for the unforced porous medium equation \eqref{pme}. Hereafter, given a compactly supported datum $u_0$, we define $\mathsf{R}(t)$ to be the radius of the smallest ball that contains the support of the solution at time $t$.


\begin{thm}\label{super p>m}
Let assumptions \eqref{H1}, \eqref{H3} be satisfied. Let $u_0\in L^{\infty}(M), u_0\geq 0$ with $\operatorname{supp}\, u_0\subset B_{R_0}$ for some $R_0>0$. Suppose that $p > m$ and that $\|u_0\|_\infty$ is sufficiently small in a sense to be made more precise below. Then problem \eqref{eq1} (with $T=\infty$) has a global in time solution $u(t)$. Besides, the bound
\[
u(x,t)\le C \zeta(t) \left[ 1- \frac r a \eta(t) \right]_+^{\frac 1{m-1}}\,\qquad \forall t\ge0,\ \forall x\in M
\]
holds for the following choices of functions $\eta, \zeta$, of the constants $C>0, a>0$, and of the initial data $u_0$:
\begin{enumerate}[(i)]
\item \[\zeta(t):=(\tau+t)^{-\frac1{m-1}}[\log(\tau +t)]^{\frac{\beta}{m-1}}\,, \quad \eta(t):=[\log(\tau +t)]^{-\beta} \]
for any given $\beta \geq 1$. The constants $a= a_1(h,N,m,R_0)$, $\tau=\tau_1(m,\beta,p,a)$ must be large enough and $C^{m-1}=c\,a^2$ for a suitable $c=c(\beta,m)$. Finally, one requires that $\|u_0\|_\infty\le C_1(h,N, m,\beta,p, R_0)$ is sufficiently small.

As a consequence one has, for the class of data considered and all $t\ge0$:
\begin{equation}\label{fb}
\mathsf{R}(t)\le a\,[\log(\tau +t)]^{\beta},\qquad \|u(t)\|_\infty\le C\,(\tau+t)^{-\frac1{m-1}}[\log(\tau +t)]^{\frac{\beta}{m-1}};
\end{equation}
\medskip
\item \[\zeta(t):=(\tau+t)^{-\alpha}\,, \quad \eta(t):= (\tau +t)^{-\beta} \]
with
\[
\frac 1{p-1} < \alpha <\frac 1{m-1},  \qquad \beta =1-\alpha(m-1)\,.
\]
The constants $a= a_2(h,N,m,R_0)$, $\tau=\tau_2(m,\alpha,p,a)$ must be large enough and $C^{m-1}=c\,a^2$ for a suitable $c=c(\alpha,m)$. Finally, one requires that $\|u_0\|_\infty\le C_2(h,N, m,\beta,p, R_0)$ is sufficiently small, with $C_2>C_1$ given in item \it(i).

As a consequence one has, for the class of data considered and all $t\ge0$:
$$
\mathsf{R}(t)\le a\,(\tau +t)^{\beta},\qquad \|u(t)\|_\infty\le C\,(\tau+t)^{-\alpha}.
$$

\end{enumerate}
\end{thm}

The above theorem provides upper bounds on solutions, and hence on the corresponding $L^\infty$ norm and free boundary radius $\mathsf{R}(t)$, that depend on the parameters involved and hence on the class of data considered. Of course, rougher bounds correspond to a wider set of initial data. It is important to stress that both bounds appearing in \eqref{fb}, in the case $\beta=1$, correspond exactly to those valid for the free porous medium equation on $\mathbb{H}^N$ proved in \cite{VazH} and developed upon in \cite{GMV}, which are known to be sharp. We do not know whether the rougher upper bounds stated above are optimal for some class of data, but we complement the results by showing that, under \it lower \rm curvature bounds, blow-up can occur for sufficiently big, although compactly supported, data.

In our next result we show in fact that if $p>m$ and lower curvature bounds (see \eqref{H2}) hold, all sufficiently large data give rise to solutions blowing up in finite time.

\begin{thm}\label{sub p>m}
Let assumptions \eqref{H1}, \eqref{H2} be satisfied. Suppose that $p > m$. For any $T>0$ there exist compactly supported initial data $u_0\in L^{\infty}(M), u_0\geq 0$ such that the corresponding solution $u(t)$ of problem \eqref{eq1} blows up at a time $S\le T$ in the sense that $\|u(t)\|_\infty\to+\infty$ as $t\to S^-$. More precisely, the bound
\[
u(x,t)\ge C \zeta(t) \left[ 1- \frac r a \eta(t) \right]_+^{\frac 1{m-1}}\,\qquad \forall t\in(0,T\wedge S),\ \forall x\in M
\]
holds for the following choices of functions $\eta, \zeta$, of the constants $C,a,T>0$, and of the class of initial data $u_0$:
\begin{enumerate}[(i)]
\item
\[
\zeta(t):=(T-t)^{-\alpha}[-\log(T -t)]^{\frac{\beta}{m-1}}\,, \qquad \eta(t):=[-\log(T -t)]^{-\beta} \,\,\; \text{for every}\;\; t\in [0,T) \, ,
\]
with $ T \in (0,1) $ and
\[
\alpha > \frac 1 {m-1} \, , \ \beta >0 \qquad \text{or } \qquad \alpha=\frac 1{m-1} \, , \ 0<\beta\leq 1 \, .
\]
The constant $C=C(a,\alpha, \beta,m,k, N,C_0,p)$ must be large enough ($ C_0 $ is as in \eqref{eqC0})
and that $T = T(a, C, p, m, \alpha, \beta, N, k, C_0) \in (0,1)$ is small enough. Finally one requires that $\operatorname{supp}\, u_0\supset B_{R_0}$ with $R_0=R_0(a,T,\beta)$ large enough and $\inf_{B_{R_0}}u_0\ge K(C,T,m,\alpha,\beta)$ large enough;
\vskip12pt
\item \[\zeta(t):=(T-t)^{-\alpha} \, , \quad \eta(t):=(T -t)^{\beta} \quad \text{for every } t\in [0, T) \, , \]
with
\[
\alpha > \frac 1{m-1} \, , \qquad 0<\beta\leq \frac{\alpha(m-1)-1}{2} \, .
\]
The constant $C=C(a,\alpha, \beta,m,k, N,C_0,p)$ must be large enough and that $T=T(a, C,$ $p, m, \alpha, \beta, N, k, C_0) \in (0,1)$ is small enough. Finally one requires that $\operatorname{supp}\, u_0\supset B_{R_0}$ with $R_0=R_0(a,T,\beta)$ large enough and $\inf_{B_{R_0}}u_0\ge K(C,T,\alpha)$ large enough.

\end{enumerate}
\end{thm}

We comment that the above result will be shown by constructing appropriate subsolutions that blow up \it everywhere \rm in $M$ at time $T$, with support becoming the whole $M$ exactly at time $T$. This does not rule out the possibility that $u$ blows up locally in the $L^\infty$ norm at some \it earlier \rm time $S$.

We now discuss the case $p<m$. As concerns global existence results, we are forced to restrict ourselves to the range $1<p\le(1+m)/2$. In that range, and under an upper bound on curvature (see \eqref{H3}), all compactly supported data give rise to a global in time solution.

\begin{thm}\label{super p<m}
Let assumptions \eqref{H1}, \eqref{H3} be satisfied. Let $u_0\in L^{\infty}(M), u_0\geq 0$ with $\operatorname{supp}\, u_0\subset B_{R_0}$ for some $R_0>0$.  Suppose that
\[
1<p\le\frac{m+1}{2}.
\]
Then problem \eqref{eq1} (with $T=\infty$) has a global in time solution $u(t)$. More precisely the bound
\[
u(x,t)\le C \zeta(t) \left[ 1- \frac r a \eta(t) \right]_+^{\frac 1{m-1}}\,\qquad \forall t\ge0,\ \forall x\in M
\]
holds for the following choices of functions $\eta, \zeta$, of the constants $C,a>0$:

\begin{enumerate}[(i)]

\item If $1<p<\frac{m+1}{2}$ one chooses
\[\zeta(t):=(\tau+t)^{\alpha}\,, \quad \eta(t):= (\tau +t)^{-\beta}\,, \]
with
\[
\alpha\geq \frac 1{m-2p +1},  \quad \beta =\frac{1+\alpha(m-1)}2\,,\quad \tau\geq 1\,,
\]
supposing in addition that $a\ge2R_0\vee H$ with $H=H(m,N,h,\beta)$ sufficiently large, and that $C=C(m,N,h,a,p)$ satisfies the (compatible) bounds $c_1\,a^{1/(m-p)}\le C\le c_2\, a^{2/(m-1)}$, where $c_1,c_2$ depend on $m,N,h,p$ and $\tau=\tau(C,\alpha,m,u_0)$ is sufficiently large.

\item If $p=(m+1)/2$ one chooses
\[\zeta(t):=\exp\{\alpha(\tau+t)\} \,, \quad \eta(t):=\exp\{-\beta(\tau +t)\}  \, , \]
with
\[
\alpha \ge \alpha_0(N,m,p,h) > 0 \, ,  \qquad \beta =\frac{\alpha(m-1)}2 \,, \quad \tau \ge 0,
\]
supposing in addition that the conditions on $a,C,\tau$ given in item \it i) above hold. 

\end{enumerate}

\end{thm}

In the whole range $1<p<m$ we can prove qualitatively similar lower bounds. In fact, if $p\in(1,m)$ and lower curvature bounds (see \eqref{H2}) hold, all sufficiently large data give rise to solutions blowing up at worst in infinite time.

\begin{thm}\label{sub p<m}
Let assumptions \eqref{H1}, \eqref{H2} be satisfied. Let $u_0\in L^{\infty}(M), u_0\geq 0$ with $\operatorname{supp}\, u_0\supset B_{R_0}$ for some $R_0>0$. Suppose that
\[
1<p<m\,
\]
and that
\[
0<\alpha<\frac 1{m-1} \, , \qquad \beta=\frac{\alpha(m-1)+1}2 \, .
\]
Then the bound
\[
u(x,t)\ge C \zeta(t) \left[ 1- \frac r a \eta(t) \right]_+^{\frac 1{m-1}}\,\qquad \forall t\in(0,S),\ \forall x\in M
\]
holds, $S\le+\infty$ being the maximal existence time, for the following choices of functions $\eta, \zeta$, of the constants $C,a>0$, and of the class of initial data $u_0$:
\[\zeta(t):=(\tau+t)^{\alpha}\,, \quad \eta(t):=(\tau +t)^{-\beta} \quad \text{for every } t\in [0, \infty) \, ,\]

where $C=C(a,m,k,N,\alpha, \beta,C_0)$ ($ C_0 $ is as in \eqref{eqC0}) and must be sufficiently large,
$\tau=\tau(a, C, p, m, \alpha, \beta, N, k, C_0) \ge 1 $ is sufficiently large, and finally one requires that $R_0=R_0(a,\tau,\beta)$ is large enough and that $\inf_{B_{R_0}}u_0\ge K(C,\tau,\alpha)$ is large enough.

\end{thm}

\begin{oss}\rm We stress that, in the range $1<p\le(1+m)/2$, the combination of the results given in Theorems \ref{super p<m}, \ref{sub p<m} shows that \it large \rm data give rise to solutions existing for all times but blowing up pointwise everywhere as $t\to+\infty$.
\end{oss}

\medskip
Let $M=\mathbb{H}_h^N$ be the simply connected manifold with sectional curvatures everywhere equal to $-h^2$. Then $\lambda_1(\mathbb{H}_h^N)=\frac{(N-1)^2}4h^2$. Let $v>0$ be a positive, bounded solution of the equation $-\Delta_{\mathbb{H}_h^N}v=\lambda_1v$. It is known that $v$ is radial w.r.t. a given pole and monotonically decreasing as a function of the geodesic distance. Notice that $v$ can be chosen so that $\|v\|_\infty\le1$.

\smallskip
As a final result, we show that data which are below a suitable profile related either to the equation $-\Delta u=u^q$, or to a ground state of $-\Delta$, both equations being in principle considered on $\mathbb{H}_h^n$, and the corresponding solutions being transplanted on $M$, give rise to global in time solutions when $p\ge m$. In fact, we remind the reader that the equation
\begin{equation}\label{elliptic}
-\Delta u=u^q,\qquad \textrm{on}\ \mathbb{H}_h^n
\end{equation}
admits \it strictly positive \rm solutions for all $q>1$, see \cite{BGGV}. Stationary solutions to \eqref{eq1} correspond to solutions of \eqref{elliptic} with $q=p/m$, which is larger than one iff $p>m$. Notice that positive, bounded, \it energy \rm solutions to \eqref{elliptic} do exist (and are unique up to hyperbolic translations) when $q\in\left(1,\frac{N+2}{N-2}\right)$ due to the results of \cite{MS}. We shall show that small data not necessarily with compact support and possibly with large $L^p$ norms ($p\ge1$) give rise to solutions existing globally in time.


\begin{thm}\label{eigen p>m}
Let assumptions \eqref{H1}, \eqref{H3} be satisfied. Suppose that $p\ge m$ and, in the case $p=m$ only, that radial sectional curvatures $\textrm{K}_{\omega}$ satisfy $\textrm{K}_{\omega}(x)\leq -h^2$ for all $x \in M\setminus\{o\}$, with
$h\geq 2/(N-1)$.

Let $v$ be a ground state of the Laplacian on $\mathbb{H}_h^n$ and, for $p>m$, let $V$ be a strictly positive solution to \eqref{elliptic} with $q=p/m$, and transplant such functions on $M$. Suppose that, in case $p=m$, $u_0\leq v^{\frac 1{m}}$ and that, in case $p>m$, $u_0\leq V^{\frac 1{m}}$. Then problem \eqref{eq1} (with $T=\infty$) has a global in time solution $u(t)$ that satisfies $0\le u(t)\le v^{\frac 1{m}}$ or, respectively, $0\le u(t)\le V^{\frac 1{m}}$, for all $t\ge0$.
\end{thm}

\begin{oss}
\rm By the results of \cite{BGGV} one knows that there exist, for all $q>1$, infinitely many strictly positive solutions to \eqref{elliptic}. All of them have \it polynomial decay \rm at infinity, except the unique (up to translations) energy solution. In particular, the $L^p$ norm ($p\ge1$) of data complying with the assumptions of Theorem \ref{eigen p>m} can be arbitrarily large. The same comment applies when $p=m$ since \it any \rm ground state of $-\Delta$ can be chosen. Notice that data might not have compact support provided they are positive and below the suitable stationary profile.
\end{oss}

\section{A family of supersolutions and subsolutions}\label{general}

We recall that, throughout this section, $ m>1 $ and $ p>1 $.

In order to construct a family of supersolutions and of subsolutions of equation
\begin{equation}\label{eq4}
u_t \,=\, \Delta (u^m)+ u^p \quad \textrm{in}\;\; M\times (0,T)\,,
\end{equation}
consider
two functions $\eta, \zeta\in C^1([0, T]; \mathbb R_+)$ and  two constants $C>0, a>0$. Define
\begin{equation}\label{eq14}
u(x,t)\equiv u(r(x), t):= C \zeta(t) \left[ 1- \frac r a \eta(t) \right]_+^{\frac 1{m-1}}\,.
\end{equation}
For further references, we compute $$u_t - \Delta(u^m) -u^p .$$
To this aim, set
\[ F(r,t):=1- \frac r a \eta(t)\,, \]
\[\mathcal D:=\{(x, t) \in (M\setminus\{o\}) \times (0, T)\, |\,  0<F(r, t) <1\}\,.\]
For any $(x,t)\in \mathcal D$ we have
\begin{equation}\label{eq15}
\begin{aligned}
u_t(r, t)&=C \zeta'(t) F^{\frac 1{m-1}} - \frac C{m-1}\zeta(t) F^{\frac 1{m-1}-1}\frac{\eta'(t)}{\eta(t)}\frac{r}{a} \eta(t) \\
&= C \zeta'(t) F^{\frac 1{m-1}} - \frac{C}{m-1}\zeta(t)\frac{\eta'(t)}{\eta(t)}F^{\frac 1{m-1}-1}+\frac C{m-1}\zeta(t)\frac{\eta'(t)}{\eta(t)}F^{\frac 1{m-1}}\,;
\end{aligned}
\end{equation}
\begin{equation}\label{eq16}
u^m_r(r, t)=-\frac{C^m m}{a(m-1)}\zeta^m(t) \eta(t) F^{\frac 1{m-1}}\,;
\end{equation}
\begin{equation}\label{eq17}
u^m_{rr}(r, t)=\frac{C^m m}{ a^2(m-1)^2} \zeta^m(t) \eta^2(t) F^{\frac 1{m-1}-1}\,.
\end{equation}
By \eqref{eq3}, \eqref{eq15}--\eqref{eq17},
\begin{equation}\label{eq18}
\begin{aligned}
 u_t - \Delta(u^m) - u^p & =C F^{\frac 1{m-1}-1}\Big\{F\left[\zeta'(t) +\frac{C^{m-1}m}{a(m-1)}\zeta^m(t) \eta(t) \m(r, \theta) + \frac{\zeta(t)}{m-1}\frac{\eta'(t)}{\eta(t)}\right] \\
&-\frac{\zeta(t)}{m-1}\frac{\eta'(t)}{\eta(t)} - \frac{C^{m-1}m}{a^2(m-1)^2}\zeta^m(t) \eta^2(t) - C^{p-1}\zeta^p(t) F^{\frac{p-2+m}{m-1}}\Big\} \quad \textrm{in}\;\; \mathcal D.
\end{aligned}
\end{equation}

\begin{pro}[Supersolution conditions]\label{supersol}
Let assumptions \eqref{H1}, \eqref{H3} be satisfied. Let $T\in (0, \infty], \zeta, \eta \in C^1([0, T); \mathbb R_+)$. If, for all $t\in (0, T),$
\begin{equation}\label{eq19}
-\frac{\eta'(t)}{\eta^3(t)} \geq \frac{ C^{m-1}m}{a^2(m-1)}\zeta^{m-1}(t)
\end{equation}
and
\begin{equation}\label{eq20}
\zeta'(t) + \frac{C^{m-1}m}{a(m-1)}\zeta^m(t) \eta(t) \left[(N-1) h  -\frac{\eta(t)}{a(m-1)}\right]\geq C^{p-1}\zeta^p(t)\,,
\end{equation}
then $u$ as defined in \eqref{eq14} is a weak supersolution of equation \eqref{eq4}.
\end{pro}
\begin{proof} In view of \eqref{H1}, \eqref{H3} \eqref{eq13}, \eqref{eq18} and the fact that $ u $ is radially decreasing, for any $(x,t)\in \mathcal D$ we get
\begin{equation}\label{eq21}
u_t - \Delta(u^m) - u^p \geq C F^{\frac 1{m-1}-1}\left\{\xi(t) F - \delta(t) -\gamma(t) F^{\frac{p-2+m}{m-1}}\right\}\,,
\end{equation}
where
\[
\xi(t):=\zeta'(t) + \frac{C^{m-1}m}{a(m-1)}(N-1) h  \zeta^m(t) \eta(t) + \frac{\zeta(t)}{m-1}\frac{\eta'(t)}{\eta(t)}\,,
\]
\begin{equation}\label{eq23}
\delta(t):=   \frac{\zeta(t)}{m-1}\frac{\eta'(t)}{\eta(t)} + \frac{C^{m-1}m}{a^2(m-1)^2} \zeta^m(t) \eta^2(t) \,,
\end{equation}
\begin{equation}\label{eq24}
\gamma(t):=C^{p-1}\zeta^p(t)\,.
\end{equation}
For every $ t \in (0,T) $ and $ F \in [0,1] $, let us define
\[\varphi(F,t):= \xi(t) F - \delta(t) - \gamma(t) F^{\frac{p-2+m}{m-1}}\,. \]
Note that \eqref{eq19} implies
\[
\varphi(0,t)\geq 0 \qquad \textrm{for every } t \in (0,T)  \, ,
\]
whereas \eqref{eq20} implies
\[
\varphi(1,t) \geq 0 \qquad \textrm{for every } t \in (0,T)  \, .
\]
Therefore, since $F\mapsto \varphi(F,t)$ is concave (recall that $ p,m>1 $),
\begin{equation}\label{eq27}
\varphi(F,t) \geq 0 \quad \textrm{for every } 0 \le F \le 1 \quad \textrm{and } t \in (0,T) \, .
\end{equation}
Thus, because for each $(x,t)\in \mathcal D$ there holds $ 0<F(x,t)<1$, due to \eqref{eq27} and \eqref{eq21} we deduce that
\[ u_t - \Delta(u^m) - u^p \geq 0\quad \textrm{in}\;\, \mathcal D\,.\]
Now observe that $u\in C(M\times [0, T))$, $ u^m \in C^1((M \setminus \{ o \} )\times [0,T) ) $ (recall \eqref{eq16}) and, by the definition of $u$,
\[ u \equiv 0 \quad \textrm{in}\,\, M \setminus \mathcal D \setminus [\{o\}\times (0, T)] \, .\]
Hence,
\[ u_t - \Delta(u^m) - u^p \geq 0\quad \textrm{in}\;\,(M\setminus\{o\})\times (0, T)\]
in the weak sense. On the other hand, thanks to a standard Kato-type inequality (note that $ u^m_r(0,t) \le 0 $), we can easily infer that
\[ u_t - \Delta(u^m) - u^p \geq 0 \quad \textrm{weakly in } M\times (0, T) \, .\]
\end{proof}

In order to construct subsolutions, we need to introduce some preliminary materials. Let
\begin{equation}\label{eq28}
\sigma(t):=\zeta'(t) + \frac{C^{m-1}m}{a(m-1)}(N-1) k \coth(k)  \zeta^m(t) \eta(t) + \frac{\zeta(t)}{m-1}\frac{\eta'(t)}{\eta(t)}\,,
\end{equation}
\begin{equation}\label{eq28c}
\delta_0(t):= \frac{\zeta(t)}{m-1}\frac{\eta'(t)}{\eta(t)}\,,
\end{equation}
and
\begin{equation}\label{eq28b}
\sigma_0(t):=\zeta'(t) + \frac{C^{m-1}m}{a(m-1)}(N-1) C_0  \zeta^m(t) \eta(t) + \frac{\zeta(t)}{m-1}\frac{\eta'(t)}{\eta(t)}\,,
\end{equation}
where
\begin{equation}\label{eqC0}
(N-1)\,C_0\geq 1+ \max_{(r,\theta)\in [0,1]\times\mathbb S^{N-1}} \m(r, \theta) \, r \, .
\end{equation}
Note that such a $C_0>0$ does exist since $M$ is locally Euclidean, i.e.~$ \m(r,\theta) \sim \tfrac{N-1}{r} $ as $ r \to 0 $.

Let us set
\begin{equation}\label{eq43}
w(x,t)\equiv w(r(x), t):=
\begin{cases}
u(x,t) & \textrm{in}\;\; (M\setminus B_1)\times (0,T)\,, \\
v(x,t) & \textrm{in}\;\; B_1\times (0, T)\,,
\end{cases}
\end{equation}
where
\begin{equation}\label{eq34}
v(x,t)\equiv v(r(x), t):=C \zeta(t)\left[1-\frac{\eta(t)}{2a}(r^2+1)\right]_+^{\frac 1{m-1}}, \quad  (x, t)\in B_1\times [0, T)\,.
\end{equation}
Notice that $w^m$ is of class $C^1$.

\begin{pro}[Subsolution conditions]\label{subsol} Let assumptions \eqref{H1}, \eqref{H2} be satisfied. Let $T\in (0, \infty], \zeta, \eta\in C^1([0, T); \mathbb R_+)$ with 
\begin{equation}\label{eq40c}
0<\eta(t)\leq \frac a 2 \quad \textrm{for all}\,\, t\in (0, T)\,.
\end{equation}
Let $\sigma, \delta, \gamma, \sigma_0, \delta_0$ be defined by \eqref{eq28}, \eqref{eq23}, \eqref{eq24}, \eqref{eq28b} and, respectively, \eqref{eq28c}. Assume that, for all $t\in (0, T),$
\begin{equation}\label{eq29}
\left[\left( \frac{m-1}{p-2+m}\right)^{\frac{m-1}{p-1}} -\left(\frac{m-1}{p-2+m}\right)^{\frac{p-2+m}{p-1}}\right] \sigma_+^{\frac{p-2+m}{p-1}}(t) \leq \delta(t) \, \gamma^{\frac{m-1}{p-1}}(t) \,,
\end{equation}
\begin{equation}\label{eq30}
(m-1) \sigma(t) \leq (p-2+m) \gamma(t)\,,
\end{equation}
\begin{equation}\label{eq41}
2^{\frac{p-2+m}{m-1}}[\sigma_0(t) - \delta_0(t)] \leq \gamma(t)\,.
\end{equation}
Then $w$ as defined in \eqref{eq43} is a weak subsolution of equation \eqref{eq4}.
\end{pro}
\begin{proof}
Let $u$ be as in \eqref{eq14}, and set
$$
\mathcal E :=\left\{ (x,t) \in (M\setminus B_1) \times (0, T) : \, 0<F(r,t)<1  \right\} .
$$
In view of \eqref{H1}, \eqref{H2}, \eqref{eq3}, \eqref{eq12} and again the fact that $u$ is radially decreasing, we deduce that
\begin{equation}\label{eq31}
u_t - \Delta(u^m) - u^p \leq C F^{\frac 1{m-1}-1}\left\{\sigma(t) F - \delta(t) -\gamma(t) F^{\frac{p-2+m}{m-1}}\right\} ,
\end{equation}
Given \eqref{eq31}, we can suppose with no loss of generality that $ \sigma(t) \ge 0 $ for all $ t \in (0,T) $. Let
\[\varphi_0(F,t):= \sigma(t) F -\delta(t) - \gamma(t) F^{\frac{p-2+m}{m-1}} \quad \textrm{for all } F \in [0,1] \ \textrm{and} \ t \in (0,T) \, .\]
Observe that (for better readability from now on we omit time dependence)
\[
\frac{\partial \varphi_0}{\partial F} (F,t) =\sigma - \frac{p-2+m}{m-1} \, \gamma \, F^{\frac{p-1}{m-1}} \, ;
\]
as a consequence,
$$
\frac{\partial \varphi_0}{\partial F} (F,t)=0 \quad \textrm{if and only if} \,\, F=F_0:=\left(\frac{m-1}{p-2+m}\frac{\sigma}{\gamma}\right)^{\frac{m-1}{p-1}},
$$
and $F_0$ is the maximum point of the (concave) function $ F \mapsto \varphi_0(F,t)$. Thanks to \eqref{eq30}, $0 \le F_0 \le 1$. Moreover, an explicit computation shows that
\begin{equation}\label{eq32}
\varphi_0(F_0,t)=\frac{\sigma^{\frac{p-2+m}{p-1}}}{\gamma^{\frac{m-1}{p-1}}}\left[\left( \frac{m-1}{p-2+m}\right)^{\frac{m-1}{p-1}} -\left(\frac{m-1}{p-2+m}\right)^{\frac{p-2+m}{p-1}}\right] -\delta \, .
\end{equation}
From \eqref{eq29} and \eqref{eq32} we obtain
$$\varphi_0(F_0)\leq 0 $$
which, combined with \eqref{eq31}, yields
\[
u_t - \Delta(u^m) - u^p \leq 0 \quad \textrm{in } \mathcal E \, .
\]
Since $u\in C(M\times [0, T))$, $ u^m \in C^1((M \setminus \{ o \} )\times [0,T) ) $ and, by the definition of $u$,
\[ u \equiv 0 \quad \textrm{in } M \setminus B_1 \setminus \mathcal E \, , \]
there holds
\begin{equation}\label{eq33b}
u_t - \Delta(u^m) - u^p \leq 0 \quad \textrm{weakly in } (M\setminus B_1) \times (0, T) \, .
\end{equation}

\medskip
Now let $v$ be as in \eqref{eq34}. Set
$$
\mathcal P :=\{x\in B_1\times (0, T): \, 0<G(r, t)<1\} \, ,
$$
where the function $G$ is defined as
\[G(r,t):=1-\frac{\eta(t)}{2a}(r^2+1) \, . \]
For any $(x,t)\in \mathcal P $, we have:
\begin{equation}\label{eq35}
v_t(r, t)=C \zeta'(t) G^{\frac 1{m-1}}-\frac{C}{m-1}\zeta(t) \frac{\eta'(t)}{\eta(t)}G^{\frac 1{m-1}-1} + \frac{C}{m-1}\zeta(t)\frac{\eta'(t)}{\eta(t)}G^{\frac 1{m-1}}\,;
\end{equation}
\[
v^m_r(r,t)=-\frac{C^m m}{a(m-1)}\zeta^m(t)\eta(t) r G^{\frac 1{m-1}}\,;
\]
\begin{equation}\label{eq37}
\begin{aligned}
v^m_{rr}(r,t)&=-\frac{C^m m}{a(m-1)}\eta(t)\zeta^m(t) G^{\frac 1{m-1}} + \frac{C^m m}{a^2(m-1)^2} \zeta^m(t) \eta^2(t) r^2 G^{\frac1{m-1}-1}\\
&\geq -\frac{C^m m}{a(m-1)}\eta(t)\zeta^m(t) G^{\frac 1{m-1}} \,.
\end{aligned}
\end{equation}
In view of \eqref{eqC0}, \eqref{eq3} and \eqref{eq35}--\eqref{eq37}, we deduce that
\begin{equation}\label{eq38}
v_t - \Delta(u^m) - v^p \leq C \, G^{\frac 1{m-1}-1}\left\{\sigma_0(t) G - \delta_0(t) -\gamma(t) G^{\frac{p-2+m}{m-1}}\right\} .
\end{equation}
Due to \eqref{eq40c}, for each $(x,t)\in \mathcal P $ there holds
\[ \frac 1 2 \leq G(r,t) \leq 1\,.\]
So, \eqref{eq38} and \eqref{eq41} yield
\begin{equation}\label{eq39}
v_t - \Delta(v^m) - v^p \leq 0 \quad \textrm{in } \mathcal P \equiv B_1\times (0, T) \, ,
\end{equation}
in the classical sense.
Because $ w \in C ( M \times [0, T)) $ and $ w^m \in C^1( M \times [0, T)) $ (note that by construction $ u=v $ and $u^m_r=v^m_r$ on $\partial B_1\times (0, T)$), from \eqref{eq33b} and \eqref{eq39} the thesis easily follows.
\end{proof}

\section{Proofs of the main results}

We provide here complete proofs of our main results, by using explicit barrier arguments based on the results of Section \ref{general} and on the comparison results given in Section \ref{existence}.

\subsection{Supersolutions}

We now provide some \emph{explicit} supersolutions from which the results of Theorems \ref{super p>m}, \ref{super p<m}, \ref{eigen p>m} will follow.

\begin{lem}\label{supersol1}
Let assumptions \eqref{H1}, \eqref{H3} be satisfied. Let $u_0\in L^{\infty}(M), u_0\geq 0$ with $\operatorname{supp}\, u_0\subset B_{R_0}$ for some $R_0>0$. Suppose that $p > m$. Let
\[\zeta(t):=(\tau+t)^{-\alpha}[\log(\tau +t)]^{\frac{\beta}{m-1}}\,, \quad \eta(t):=[\log(\tau +t)]^{-\beta} \]
with $\alpha=\frac 1{m-1}, \beta \geq 1$. Suppose that
\begin{equation}\label{eq54}
\frac{C^{m-1}}{a^2}\leq \frac{m-1}{m}\beta\,,
\end{equation}
\begin{equation}\label{eq55}
a\geq \frac{2}{h(N-1)(m-1)} \, ,
\end{equation}
\begin{equation}\label{eq56-bis}
2 \alpha \leq \frac{(N-1) h}{2} \, \frac{C^{m-1} \, m}{a(m-1)}  \, ,
\end{equation}
and that $\tau=\tau(m,\beta,p,a) \geq e $ is large enough. Then the function $u$ defined in \eqref{eq14} is a weak supersolution of equation \eqref{eq4} with $T=\infty$. Moreover, if
\begin{equation}\label{eq61}
a \geq  2 R_0 \,, \qquad \|u_0\|_{\infty}\leq \frac{C}{2^{m-1}} \, \tau^{-\alpha} \, (\log \tau)^{\frac{\beta}{m-1}}\,,
\end{equation}
then $u$ is also a supersolution of problem \eqref{eq1} with $T=\infty$.
\end{lem}
\begin{proof}
Condition \eqref{eq19} with $T=\infty$ reads
\[\beta[\log(\tau +t)]^{\beta-1} \geq \frac{m C^{m-1}}{a^2(m-1)} (\tau +t)^{1-\alpha(m-1)} =  \frac{m C^{m-1}}{a^2(m-1)} \quad \textrm{for all } t>0 \, , \]
which holds due to \eqref{eq54} and the fact that $ \tau \ge e $. Moreover, condition \eqref{eq20} with $T=\infty$ reads
\[-\alpha (\tau +t)^{-\alpha-1}[\log(\tau +t)]^{\frac{\beta}{m-1}} + \frac{\beta}{m-1}(\tau +t)^{-\alpha-1}[\log(\tau +t)]^{\frac{\beta}{m-1}-1}  \]
\[+ \frac{C^{m-1}m}{a(m-1)}(\tau +t)^{-\alpha m}[\log(\tau +t)]^{\frac{\beta}{m-1}}\left[(N-1)h-\frac{[\log(\tau+t)]^{-\beta}}{a(m-1)}\right] \]
\[\geq C^{p-1}(\tau+t)^{-\alpha p} \, [\log(\tau +t)]^{\frac{\beta p}{m-1}} \qquad \textrm{for all } t>0 \, , \]
which is fulfilled, in view of \eqref{eq55} and \eqref{eq56-bis}, provided  $ \tau = \tau(m,\beta,p,a) \ge e $ is so large that (note that $ \alpha+1 =\alpha m $)
\begin{equation}\label{eqx10}
\begin{aligned}
\alpha \left( \tau+t \right)^{-\alpha-1} \left[ \log\left( t+\tau \right) \right]^{\frac{\beta}{m-1}} \ge C^{p-1} \left( t+\tau \right)^{-\alpha p} \left[ \log(\tau+t) \right]^{\frac{\beta p}{m-1}}  \quad \text{for all } t>0 \, ,
\end{aligned}
\end{equation}
where in the r.h.s.~one can replace $C$ with the upper bound given in \eqref{eq54}. We point out that in this last inequality the existence of such a $ \tau $ is ensured since $p>m$. Hence, in view of Proposition \ref{supersol}, we obtain that $u$ is a weak supersolution of equation \eqref{eq4}. In addition, \eqref{eq61} implies that (recall the explicit expression \eqref{eq14})
\begin{equation}\label{eq62}
u_0(x) \leq u(x,0)\quad \textrm{for all}\,\, x\in M\,.
\end{equation}
Hence $u$ is also a supersolution of problem \eqref{eq1}.

Finally, let us briefly explain how the above conditions can be made compatible: first one picks $C$ so as to satisfy \eqref{eq54} as equality, which means that $ C^{m-1} \sim a^2 $, then plugs this choice in \eqref{eq56-bis} and selects $a$ so large that both \eqref{eq56-bis} and \eqref{eq55} are met. Lastly, $ \tau \ge e $ is taken so large that \eqref{eqx10} holds upon the previous choices.
\end{proof}

\begin{lem}\label{supersol2}
Let assumptions \eqref{H1}, \eqref{H3} hold and suppose that $p>m$. Let $u_0\in L^{\infty}(M)$, $u_0\geq 0$ with $\operatorname{supp}\, u_0\subset B_{R_0}$ for some $R_0>0$.  Let
\[\zeta(t):=(\tau+t)^{-\alpha}\,, \quad \eta(t):= (\tau +t)^{-\beta} \]
with
\begin{equation}\label{eq43-bis}
\frac 1{p-1} < \alpha <\frac 1{m-1},  \qquad \beta =1-\alpha(m-1)\,.
\end{equation}
Suppose that \eqref{eq54}, \eqref{eq55}, \eqref{eq56-bis} hold
and that $\tau=\tau(m , \alpha , p, a ) \geq 1$ is large enough. Then the function $u$ defined in \eqref{eq14} is a weak supersolution of equation \eqref{eq4} with $ T=\infty $. Moreover, if
\begin{equation}\label{eq63}
a\geq 2 R_0\,,\quad \|u_0\|_{\infty}\leq \frac{C}{2^{m-1}}  \tau^{-\alpha}\,,
\end{equation}
then $u$ is also a supersolution of problem \eqref{eq1}.
\end{lem}
\begin{proof}
Condition \eqref{eq19} with $T=\infty$ reads
\[\beta (\tau + t)^{2\beta -1+\alpha(m-1)} \geq \frac{m C^{m-1}}{a^2(m-1)} \quad \textrm{for all } t>0 \, , \]
which holds for all $ \tau \ge 1 $, in view of \eqref{eq54}, providing that
\[
2\beta-1+\alpha(m-1)\geq 0\, ,
\]
the latter inequality being trivially guaranteed by \eqref{eq43-bis}. Furthermore, condition \eqref{eq20} with $T=\infty$ reads
\[-\alpha (\tau +t)^{-\alpha-1} +\frac{m C^{m-1}}{a(m-1)}(\tau+t)^{-\alpha m -\beta}\left[ (N-1)h -\frac{(\tau+t)^{-\beta}}{a(m-1)}\right] \]
\[\geq C^{p-1}(\tau+t)^{-\alpha p}\quad \textrm{for all}\,\, t>0 \,,\]
which is fulfilled, thanks to \eqref{eq55} and \eqref{eq56-bis}, if
\[
\beta-1+\alpha(m-1)\leq 0, \quad \alpha(p-m)\geq \beta > 0 \ \text{(recall \eqref{eq43-bis})}
\]
and $\tau=\tau(m, \alpha , p, a) \geq 1$ is so large that for all $t>0$
\[
\alpha \left( \tau + t \right)^{-\alpha -1} \geq C^{p-1} \left( \tau + t \right)^{-\alpha p} \, ;
\]
this is always possible thanks to the first (lower) inequality in \eqref{eq43-bis}. Hence, in view of Proposition \ref{supersol}, $u$ is a weak supersolution of equation \eqref{eq4}. The fact that $C$, $ a $ and $ \tau $ can be chosen so as to satisfy the above conditions can be justified similarly to the end of the proof of Lemma \ref{supersol1}.

Finally, \eqref{eq63} yields \eqref{eq62}, thus $u$ is also a supersolution of problem \eqref{eq1}.
\end{proof}

\bf Proof of Theorem \ref{super p>m}\rm. We use comparison with the barriers constructed in Lemmas \ref{supersol1}, \ref{supersol2} for solutions to approximating problems that involve homogeneous Dirichlet boundary conditions on balls of radius $R$ with $R\to+\infty$, see Proposition \ref{thm:cp1}. The bounds still hold in such limit and yield part \it i\rm) of the claim by using Lemma \ref{supersol1} and part \it ii\rm) of the claim by using Lemma \ref{supersol2}. It is standard although tedious to check that the conditions on the initial data considered in item \it ii\rm) give rise to a larger class than the one singled out in item \it i\rm).

\hfill $\qed$

\begin{lem}\label{supersol3}
Let assumptions \eqref{H1}, \eqref{H3} be satisfied. Let $u_0\in L^{\infty}(M), u_0\geq 0$ with $\operatorname{supp}\, u_0\subset B_{R_0}$ for some $R_0>0$.  Suppose that
\begin{equation}\label{eq50}
1<p<\frac{m+1}{2}.
\end{equation}
Let
\[\zeta(t):=(\tau+t)^{\alpha}\,, \quad \eta(t):= (\tau +t)^{-\beta}\,, \]
with
\begin{equation}\label{eq44}
\alpha\geq \frac 1{m-2p +1},  \quad \beta =\frac{1+\alpha(m-1)}2\,,\quad \tau\geq 1\,.
\end{equation}
Suppose that \eqref{eq54}--\eqref{eq55} hold and that
\begin{equation}\label{eq59}
\frac{C^{m-p}}{a}\geq \frac{2(m-1)}{m(N-1)h}\,.
\end{equation}
Then the function $u$ defined in \eqref{eq14} is a weak supersolution of equation \eqref{eq4}. Moreover, if
\begin{equation}\label{eq64}
a\geq 2 R_0 \vee H\,, \qquad C\tau^{\alpha}\geq 2^{m-1} \|u_0\|_{\infty}\,,
\end{equation}
with $H=H(m,N,h,\beta)$ sufficiently large, then $u$ is also a supersolution of problem \eqref{eq1}.
\end{lem}
\begin{proof}
Condition \eqref{eq19} with $T=\infty$ reads
\[\beta(\tau + t)^{2\beta-1-\alpha(m-1)}\geq \frac{m C^{m-1}}{(m-1) \, a^2} \qquad \text{for all } t>0 \, , \]
which is satisfied, due to \eqref{eq54} and the fact that $ \tau \ge 1 $, whenever
\begin{equation}\label{eq45}
2\beta-1-\alpha(m-1)\geq 0 \qquad \Longleftrightarrow \qquad \beta \geq \frac{1+\alpha(m-1)}{2} \, .
\end{equation}
Furthermore, condition \eqref{eq20} with $T=\infty$ becomes
\[\alpha (\tau +t)^{\alpha-1} +\frac{m C^{m-1}}{a(m-1)}(\tau+t)^{\alpha m -\beta}\left[ (N-1)h -\frac{(\tau+t)^{-\beta}}{a(m-1)}\right] \]
\[\geq C^{p-1}(\tau+t)^{\alpha p}\quad \textrm{for all}\,\, t>0 \,,\]
which is fulfilled, thanks to \eqref{eq55}, \eqref{eq59} and $ \tau \ge 1 $, providing that
and
\begin{equation}\label{eq47}
\alpha(m-p)\geq \beta\,.
\end{equation}
It is straightforwardly checked that \eqref{eq44} (and \eqref{eq50}) ensures that both \eqref{eq45} and \eqref{eq47} hold. Hence, from Proposition \ref{supersol} we get that $u$ is a supersolution of equation \eqref{eq4}.


As concerns the compatibility of the conditions involving $C$ and $a$, we just point out that \eqref{eq50} is crucial in order to guarantee that one can pick $ a $ so large that also \eqref{eq59} (in addition to \eqref{eq54}--\eqref{eq55}) holds.

Finally, \eqref{eq64} implies \eqref{eq62}, so $u$ is also a supersolution of problem \eqref{eq1}.
\end{proof}

\begin{lem}\label{supersol4}
Let assumptions \eqref{H1}, \eqref{H3} be satisfied. Let $u_0\in L^{\infty}(M), u_0\geq 0$ with $\operatorname{supp}\, u_0 \subset B_{R_0}$ for some $R_0>0$.  Suppose that
\begin{equation}\label{eq53}
1<p\leq \frac{m+1}{2}\,.
\end{equation}
Let
\[\zeta(t):=\exp\{\alpha(\tau+t)\} \,, \quad \eta(t):=\exp\{-\beta(\tau +t)\} \, , \quad \tau \ge 0 \, , \]
with
\begin{equation}\label{eq48}
\alpha \ge \alpha_0(N,m,p,h) > 0 \, ,  \qquad \beta =\frac{\alpha(m-1)}2 \, .
\end{equation}
Suppose that \eqref{eq54}, \eqref{eq55} and \eqref{eq59} hold.
Then the function $u$ defined in \eqref{eq14} is a weak supersolution of equation \eqref{eq4}. Moreover, if
\begin{equation}\label{eq65}
a\geq 2 R_0\,,\quad C\exp\{\alpha \tau\}\geq 2^{m-1} \|u_0\|_{\infty}\,,
\end{equation}
then $u$ is also a supersolution of problem \eqref{eq1}.
\end{lem}
\begin{proof}
Condition \eqref{eq19} with $T=\infty$ reads
\[\beta\exp\{[2\beta-\alpha(m-1)] (\tau + t)\}\geq \frac{m C^{m-1}}{(m-1) a^2} \qquad \textrm{for all}\,\, t>0\,, \]
which is satisfied, in view of \eqref{eq54} and the fact that $ \tau \ge 0 $, as long as
\begin{equation}\label{eq51}
2\beta - \alpha(m-1)\geq 0  \qquad \Longleftrightarrow \qquad  \beta \ge \frac{\alpha(m-1)}2 \, .
\end{equation}
Furthermore, condition \eqref{eq20} with $T=\infty$ reads
\[\alpha \exp\{ \alpha(\tau +t) \} +\frac{m C^{m-1}}{a(m-1)}\exp\{ (\alpha m-\beta)(\tau+t)\}\left[ (N-1)h -\frac{\exp\{-\beta (\tau+t)\}  }{a(m-1)}\right] \]
\[\geq C^{p-1}\exp\{\alpha p(\tau+t)\}\quad \textrm{for all}\,\, t>0 \,,\]
which is fulfilled, due to \eqref{eq55}, \eqref{eq59} and the fact that $ \tau \ge 0 $, providing that
\begin{equation}\label{eq52}
\alpha(m-p) \geq \beta \, .
\end{equation}
Observe that \eqref{eq53} and \eqref{eq48} guarantee the validity of \eqref{eq51} and \eqref{eq52}. Hence, Proposition \ref{supersol} ensures that $u$ is a supersolution of equation \eqref{eq4}.

As concerns the compatibility of the conditions involving $C$ and $a$, we point out that if \eqref{eq53} holds with \emph{strict} inequalities then the same comments as in the end of the proof of Lemma \ref{supersol3} apply. Otherwise, in the critical case $ p=\frac{m+1}{2} $, by substituting $C$ with the r.h.s.~of \eqref{eq54} in condition \eqref{eq59} one sees that the only degree of freedom left to make the inequality hold is the one given by $ \alpha $ (through $ \beta $), which should be taken sufficiently large depending on $ N,m,p,h $ (i.e.~larger than a value that we labeled $ \alpha_0 $).

Finally, from \eqref{eq65} there follows \eqref{eq62}, so $u$ is also a supersolution of problem \eqref{eq1}.
\end{proof}

\medskip \bf Proof of Theorem \ref{super p<m}\rm. We use comparison with the barriers constructed in Lemmas \ref{supersol3}, \ref{supersol4} for solutions to approximating problems that involve homogeneous Dirichlet boundary conditions on balls of radius $R$ with $R\to+\infty$, see Proposition \ref{thm:cp1}. The bounds still hold in such limit and yield part \it i\rm) of the claim by using Lemma \ref{supersol3} and part \it ii\rm) of the claim by using Lemma \ref{supersol4}. \hfill $\qed$

\bigskip

We now turn to the proof of Theorem \ref{eigen p>m}. Its statement will follow from the next result.

\begin{lem}\label{supersol5}
Let assumptions \eqref{H1}, \eqref{H3} be satisfied. Suppose that $p\ge m$ and that, in case $p=m$ only,
radial sectional curvatures $\textrm{K}_{\omega}$ satisfy $\textrm{K}_{\omega}(x)\leq -h^2$ for all $x \in M\setminus\{o\}$, with
$h\geq 2/(N-1)$.

Consider a ground state $v$ of the Laplacian on $\mathbb{H}_h^n$ and, for $p>m$, a strictly positive solution  $V$ to \eqref{elliptic} with $q=p/m$, and transplant such functions on $M$. Then $u = v^{\frac 1 m}$ when $p=m$, $u = V^{\frac 1 m}$ when $p>m$, are supersolutions of equation \eqref{eq4}. Moreover, if $u_0\leq v^{\frac 1{m}},$ or $u_0\leq V^{\frac 1{m}}$ respectively, then $u$ is also a supersolution of problem \eqref{eq1}.
\end{lem}
\begin{proof}
By the properties of $v$ recalled above and Laplacian comparison \eqref{eq13} we compute:
\[\begin{aligned}-\Delta v&=-v''-\m(r,\theta)v'\ge-v''-(N-1)h\coth(hr)v'\\ &=-\Delta_{\mathbb{H}_h^n}v=\lambda_1(\mathbb{H}_h^n)\,v\ge v
\qquad \textrm{in } M\,,  \end{aligned}\]
where we have used the known bound (see \cite{M})
\[
\lambda_1 \geq \frac{(N-1)^2}{4}h^2\,
\]
and the running curvature assumption.

Since $p= m$, the function $u:= v^{\frac 1 m}$ is a positive stationary supersolution of equation \eqref{eq4}. In fact
\[
-\Delta u^m=-\Delta v\ge v= u^p.
\]

Clearly, $u$ is also a supersolution of
problem \eqref{eq1}, provided that $u_0\leq u = v^{\frac 1 m}$ in $M$. An essentially identical proof works for the case $p>m$ by replacing $v$ with $V$.
\end{proof}

\medskip \bf Proof of Theorem \ref{eigen p>m}\rm. We use comparison with the barrier constructed in Lemma \ref{supersol5} for solutions to approximating problems that involve homogeneous Dirichlet boundary conditions on balls of radius $R$ with $R\to+\infty$, see Proposition \ref{thm:cp1}. The bounds still hold in such limit. \hfill $\qed$


\subsection{Subsolutions}

We now provide some \emph{explicit} supersolutions from which the results of Theorems \ref{sub p>m}, \ref{sub p<m} will follow.

\begin{lem}\label{subsol1}
Let assumptions \eqref{H1}, \eqref{H2} hold and assume that $p>m$.
hold. Let $u_0\in L^{\infty}(M), u_0\geq 0$ with $ \operatorname{supp}\, u_0 \supset B_{R_0} $ for some $R_0>0$. Let
\[
\zeta(t):=(T-t)^{-\alpha}[-\log(T -t)]^{\frac{\beta}{m-1}}\,, \qquad \eta(t):=[-\log(T -t)]^{-\beta} \,\,\; \text{for every}\;\; t\in [0,T) \, ,
\]
with $ T \in (0,1) $ and
\begin{equation}\label{eq128}
\alpha > \frac 1 {m-1} \, , \ \beta >0 \qquad \text{or } \qquad \alpha=\frac 1{m-1} \, , \ 0<\beta\leq 1 \, .
\end{equation}
 Suppose that
\begin{equation}\label{eq100}
\frac{C^{m-1}}{a}\geq \max\left\{ \frac{\alpha  (m-1)}{m k \coth(k) (N-1)} \, , \, \frac{\alpha  (m-1) + \beta}{m C_0 (N-1)} \, , \, \tilde C \right\} ,
\end{equation}
\begin{equation}\label{eq101}
\frac{C^{m-1}}{a^2}\geq \frac{2 \beta (m-1)}{m}\,,
\end{equation}
for a suitable constant $\tilde C=\tilde C(p, m,N, k)>0$, and that $T = T(a, C, p, m, \alpha, \beta, N, k, C_0) \in (0,1)$ is small enough ($ C_0 $ is as in \eqref{eqC0}). Then the function $u$ defined in \eqref{eq14} is a weak subsolution of equation \eqref{eq4}. Moreover, if
\begin{equation}\label{eq103}
R_0 \geq a (-\log T)^\beta\,,\quad  u_0 \geq C T^{-\alpha}(-\log T)^{\frac{\beta}{m-1}} \quad \textrm{in}\;\; B_{R_0},
\end{equation}
then $u$ is also a subsolution of problem \eqref{eq1}.
\end{lem}
\begin{proof}
We take $T=T(a,\beta)>0$ so small that \eqref{eq40c} is fulfilled. With the above choices of $ \zeta $ and $ \eta $, in view of \eqref{eq128} (for the moment we only use $ \alpha \ge \tfrac{1}{m-1} $), the first inequality of \eqref{eq100} and the fact that $ T \le 1 $, we have that (recall the definition of $ \sigma $ given by \eqref{eq28})
\begin{equation}\label{eq104}
\begin{aligned}
\sigma(t) = & \, \alpha (T-t)^{-\alpha-1} \, [-\log(T-t)]^{\frac{\beta}{m-1}} \\
 & \, + \frac{C^{m-1}m (N-1) k \coth(k) }{a(m-1)}\,(T-t)^{-\alpha m} \, [-\log(T-t)]^{\frac{\beta}{m-1}}\\
\leq & \, \frac{ 2C^{m-1}m (N-1) k \coth(k) }{a(m-1)}\,(T-t)^{-\alpha m} \, [-\log(T-t)]^{\frac{\beta}{m-1}} \, .
\end{aligned}
\end{equation}
Furthermore, upon recalling the definition of $ \delta $ given by \eqref{eq23}, thanks to \eqref{eq101} we obtain the estimate
\begin{equation}\label{eq105}
\begin{aligned}
\delta(t) = & - \frac{\beta}{m-1} \, (T-t)^{-\alpha-1} \, [-\log(T-t)]^{\frac{\beta}{m-1}-1} \\
& \, + \frac{C^{m-1}m}{a^2(m-1)^2} \, (T-t)^{-\alpha m} \, [-\log(T-t)]^{\frac{\beta(2-m)}{m-1}} \\
\geq & \, \frac{C^{m-1}m}{2 a^2(m-1)^2} \, (T-t)^{-\alpha m} \, [-\log(T-t)]^{\frac{\beta(2-m)}{m-1}}
\end{aligned}
\end{equation}
as long as $ T \in (0,1) $ is so small that
\[ (T-t)^{-\alpha-1} \, [-\log(T-t)]^{\frac{\beta}{m-1}-1} \leq  (T-t)^{-\alpha m} \, [-\log(T-t)]^{\frac{\beta(2-m)}{m-1}} \qquad \forall t \in (0,T) \,.\]
Note that such a choice of $T$ is always feasible thanks to \eqref{eq128}. Now set
\begin{equation}\label{eq112}
K_1:=\left( \frac{m-1}{p-2+m}\right)^{\frac{m-1}{p-1}} -\left(\frac{m-1}{p-2+m}\right)^{\frac{p-2+m}{p-1}} > 0 \,.
\end{equation}
Due to \eqref{eq104} and \eqref{eq105}, condition \eqref{eq29} is implied by
\begin{equation}\label{eq106}
\begin{aligned}
& \, 2 K_1^{\frac{p-1}{p-2+m}}\frac{ C^{m-1}m (N-1) k \coth(k) }{a(m-1)}(T-t)^{-\alpha m}[-\log(T-t)]^{\frac{\beta}{m-1}} \\
\leq \, & C^{\frac{(p-1)(m-1)}{p-2+m}}\left(\frac{C^{m-1}m}{2 a^2(m-1)^2}\right)^{\frac{p-1}{p-2+m}}  (T-t)^{\frac{-\alpha p(m-1) -\alpha m(p-1)}{p-2+m}}[-\log(T-t)]^{\frac{\beta}{m-1}} \quad \forall t \in (0,T) \, .
\end{aligned}
\end{equation}
Note that
\[\alpha m \leq \alpha p \, \frac{m-1}{p-2+m} + \alpha m \, \frac{p-1}{p-2+m}\]
if and only if
\[(p-m)(m-1)\geq 0 \, , \]
which trivially holds since $p>m$. Hence, in view of the third inequality in \eqref{eq100}, we have that \eqref{eq106} (and so \eqref{eq29}) is fulfilled: we point out that, for this purpose, the hypothesis $ p>m $ is essential (at $ p=m $ the dependence on $C$ and $a$ vanishes and there is no more degree of freedom to make \eqref{eq106} hold). Moreover, from \eqref{eq104} we deduce that \eqref{eq30} is satisfied whenever
\begin{equation}\label{e107bis}
\frac{2 m k \coth(k)  (N-1)}{p-2+m} \leq a \, C^{p-m} \, (T-t)^{-\alpha(p-m)} \, [-\log(T-t)]^{\frac{\beta(p-1)}{m-1}} \qquad \forall t \in (0,T) \, ,
\end{equation}
and to this aim it is enough to choose $T=T(a, C, p, m, \alpha, \beta, N, k)>0$ small enough.

Finally, thanks to the middle inequality in \eqref{eq100}, we have that (recall that $ \sigma_0 $ and $ \delta_0 $ are defined by \eqref{eq28b} and \eqref{eq28c}, respectively)
\[
\sigma_0(t) - \delta_0(t) \leq \frac{ 2C^{m-1}m (N-1) C_0}{a(m-1)} \, (T-t)^{-\alpha m} \, [-\log(T-t)]^{\frac{\beta}{m-1}} \, .
\]
We therefore deduce that inequality \eqref{eq41} is satisfied provided
\begin{equation}\label{eq107}
\frac{2^{\frac{p-2+m}{m-1}+1}m (N-1) C_0}{(m-1)} \leq a \, C^{p-m} \, (T-t)^{-\alpha(p-m)} \, [-\log(T-t)]^{\frac{\beta (p-1)}{m-1}} \qquad \forall t \in (0,T) \,.
\end{equation}
Similarly to \eqref{e107bis}, it is plain that \eqref{eq107} holds if $T=T(a, C, p, m, \alpha, \beta, N, C_0)>0$ is small enough. Since we have established that \eqref{eq29}, \eqref{eq30} and \eqref{eq41} hold, from Proposition \ref{subsol} we get that $u$ is a subsolution of equation \eqref{eq4}. Furthemore, \eqref{eq103} implies that
\begin{equation}\label{eq108}
u(x,0) \leq u_0(x) \quad \text{for all } x\in M \, ,
\end{equation}
so that $u$ is also a subsolution of problem \eqref{eq1}.
\end{proof}

\begin{lem}\label{subsol2}
Let assumptions \eqref{H1}, \eqref{H2} hold and suppose that $p>m$. Let $u_0\in L^{\infty}(M)$, $u_0\geq 0$ with $\operatorname{supp}\, u_0\supset B_{R_0}$ for some $R_0>0$.
Let
\[\zeta(t):=(T-t)^{-\alpha} \, , \quad \eta(t):=(T -t)^{\beta} \quad \text{for every } t\in [0, T) \, , \]
with
\begin{equation}\label{eq127}
\alpha > \frac 1{m-1} \, , \qquad 0<\beta\leq \frac{\alpha(m-1)-1}{2} \, .
\end{equation}
Suppose that
\eqref{eq100}--\eqref{eq101} hold, and that $T=T(a, C, p, m, \alpha, \beta, N, k, C_0) \in (0,1)$ is small enough ($ C_0 $ is as in \eqref{eqC0}). Then the function $u$ defined in \eqref{eq14} is a weak subsolution of equation \eqref{eq4}. Moreover, if
\begin{equation}\label{eq115}
R_0 \geq a T^{-\beta}\,,\quad  u_0 \geq C T^{-\alpha} \quad \text{in } B_{R_0} \, ,
\end{equation}
then $u$ is also a subsolution of problem \eqref{eq1}.
\end{lem}
\begin{proof}
We take $T=T(a,\beta) \in (0,1)$ so small that \eqref{eq40c} is fulfilled. In view of the first inequality in \eqref{eq100} and \eqref{eq127} (for the moment we only use the fact that $ \beta \le \alpha(m-1)-1 $), we have (recall that $\sigma$ is defined in \eqref{eq28})
\begin{equation}\label{eq110}
\begin{aligned}
\sigma(t)= & \, \alpha(T-t)^{-\alpha-1} +  \frac{C^{m-1}m (N-1) k \coth(k) }{a(m-1)} \, (T-t)^{-\alpha m +\beta}-\frac{\beta}{m-1} \, (T-t)^{-\alpha-1}\\
\leq &  \, \frac{2 C^{m-1}m (N-1) k \coth(k) }{a(m-1)} \, (T-t)^{-\alpha m +\beta} \, .
\end{aligned}
\end{equation}
Moreover, thanks to \eqref{eq101} and \eqref{eq127} (recall that $ \delta $ is defined in \eqref{eq23}),
\begin{equation}\label{eq111}
\begin{aligned}
\delta(t) = -\frac{\beta}{m-1}\,(T-t)^{-\alpha-1} + \frac{C^{m-1} m }{a^2 (m-1)^2}\,(T-t)^{-\alpha m +2 \beta} \geq \frac{C^{m-1}m}{2 a^2 (m-1)^2}\,(T-t)^{-\alpha m + 2\beta}\,;
\end{aligned}
\end{equation}
note that here we do need that $ \beta \le \tfrac{\alpha(m-1)-1}{2} $. Now let $K_1$ be defined as in \eqref{eq112}. By virtue of \eqref{eq110} and \eqref{eq111}, condition \eqref{eq29} is implied by
\begin{equation}\label{eq113}
\begin{aligned}
&  \, 2 K_1^{\frac{p-1}{p-2+m}} \, \frac{ C^{m-1}m (N-1) k \coth(k) }{a(m-1)} \, (T-t)^{-\alpha m+\beta}\\
\leq & \,  C^{\frac{(p-1)(m-1)}{p-2+m}}\left(\frac{C^{m-1}m}{2 a^2(m-1)^2}\right)^{\frac{p-1}{p-2+m}} (T-t)^{\frac{-\alpha p(m-1)+ (2\beta-\alpha m)(p-1)}{p-2+m}} \qquad \forall t \in (0,T) \, .
\end{aligned}
\end{equation}
Observe that
\[\alpha m - \beta \leq \alpha p \frac{m-1}{p-2+m} + (\alpha m - 2 \beta)\frac{p-1}{p-2+m} \]
holds if and only if
\[(p-m)[\alpha(m-1)-\beta]\geq 0 \, , \]
which is guaranteed since $p>m$ and \eqref{eq127} holds. Therefore, from the third inequality in \eqref{eq100} we infer that \eqref{eq113} is fulfilled: we point out, once again, that here it is essential that $ p>m $, for the same reasons as in the proof of Lemma \ref{subsol1}. On the other hand, from \eqref{eq110} we deduce that \eqref{eq30} is satisfied provided
\[
\frac{2 m k \coth(k)  (N-1)}{p-2+m} \leq a \, C^{p-m} \, (T-t)^{-\alpha(p-m)-\beta} \qquad \forall t \in (0,T) \, ;
\]
to this end, it suffices to pick $T=T(a, C, p, m, \alpha, \beta, N, k)>0$ small enough.

Furthermore (recall that $ \sigma_0 $ and $ \delta_0 $ are defined by \eqref{eq28b} and \eqref{eq28c}, respectively), thanks to the central inequality in \eqref{eq100} (actually here one can replace $ \beta $ with $0$ in such inequality), we deduce that
\[
\sigma_0(t) - \delta_0(t) \leq \frac{ 2C^{m-1}m (N-1) C_0}{a(m-1)} \, (T-t)^{-\alpha m + \beta}\,.
\]
We therefore infer that inequality \eqref{eq41} is met provided
\begin{equation}\label{eq114}
\frac{2^{\frac{p-2+m}{m-1}+1}m (N-1) C_0}{(m-1)} \leq a \, C^{p-m} \, (T-t)^{-\alpha(p-m)-\beta} \qquad \forall t \in (0,T) \,.
\end{equation}
It is apparent that \eqref{eq114} is satisfied if $T=T(a, C, p, m, \alpha, \beta, N, C_0)>0$ is small enough. Since \eqref{eq29}, \eqref{eq30} and \eqref{eq41} hold, from Proposition \ref{subsol} we get that $u$ is a subsolution of equation \eqref{eq4}. Finally, \eqref{eq115} yields \eqref{eq108}, so $u$ is also a subsolution of problem \eqref{eq1}.
\end{proof}

\bf Proof of Theorem \ref{sub p>m}\rm. We use comparison with the barriers constructed in Lemmas \ref{subsol1}, \ref{subsol2}, see Proposition \ref{thm:cp2}. This yield part \it i\rm) of the claim by using Lemma \ref{subsol1} and part \it ii\rm) of the claim by using Lemma \ref{subsol2}. \hfill $\qed$

\begin{lem}\label{subsol3}
Let assumptions \eqref{H1}, \eqref{H2} be satisfied. Let $u_0\in L^{\infty}(M), u_0\geq 0$ with $\operatorname{supp}\, u_0\supset B_{R_0}$ for some $R_0>0$. Suppose that
\begin{equation}\label{eq117}
1<p<m\,.
\end{equation}
Let
\[\zeta(t):=(\tau+t)^{\alpha}\,, \quad \eta(t):=(\tau +t)^{-\beta} \quad \text{for every } t\in [0, \infty) \, .\]
Suppose that \eqref{eq101} holds,
\begin{equation}\label{eq122b}
0<\alpha<\frac 1{m-1} \, , \qquad \beta=\frac{\alpha(m-1)+1}2 \, ,
\end{equation}
\begin{equation}\label{eq100b}
\frac{C^{m-1}}{a}\geq \max\left\{ \frac{\alpha  (m-1)}{m k \coth(k)  (N-1)} \, , \, \frac{\alpha  (m-1)}{m C_0 (N-1)}  \right\}
\end{equation}
and that $\tau=\tau(a, C, p, m, \alpha, \beta, N, k, C_0) \ge 1 $ is sufficiently large. Then the function $u$ defined in \eqref{eq14} is a weak subsolution of equation \eqref{eq4}. Moreover, if
\begin{equation}\label{eq126}
R_0 \geq a \tau^{\beta}, \qquad u_0 \geq C \tau^\alpha\quad \text{in}\;\; B_{R_0}\,,
\end{equation}
then $u$ is also a subsolution of problem \eqref{eq1}.
\end{lem}

\begin{proof}
We take $\tau^\beta \geq \frac 2 a$, so \eqref{eq40c} is fulfilled. In view of \eqref{eq122b}, $ \tau \ge 1 $ and the first inequality in \eqref{eq100b}, we have that (recall that $ \sigma $ is defined in \eqref{eq28})
\begin{equation}\label{eq121}
\begin{aligned}
\sigma(t)= & \, \alpha(\tau + t)^{\alpha-1} +  \frac{C^{m-1}m (N-1) k\coth(k)  }{a(m-1)}(\tau +t)^{\alpha m -\beta}-\frac{\beta}{m-1}(\tau+t)^{\alpha-1}\\
\leq & \, \frac{2 C^{m-1}m (N-1) k \coth(k) }{a(m-1)}(\tau+t)^{\alpha m -\beta} \, .
\end{aligned}
\end{equation}
Here we are only using the fact that $ \beta \le \alpha(m-1)+1 $. Moreover, thanks to \eqref{eq101}, \eqref{eq122b}, \eqref{eq100b} and $ \tau \ge 1 $ (recall that $ \delta $ is defined in \eqref{eq23}),
\begin{equation}\label{eq123}
\begin{aligned}
\delta(t) = \, & -\frac{\beta}{m-1}(\tau+t)^{\alpha-1} + \frac{C^{m-1} m }{a^2 (m-1)^2}(\tau+t)^{\alpha m -2 \beta} \\
\geq \, & \frac{C^{m-1}m}{2 a^2 (m-1)^2}(\tau+t)^{\alpha m - 2\beta}\,.
\end{aligned}
\end{equation}

Let $K_1$ be defined as in \eqref{eq112}. Due to \eqref{eq121} and \eqref{eq123}, condition \eqref{eq29} is implied by
\begin{equation}\label{eq124}
\begin{aligned}
& \, 2 K_1^{\frac{p-1}{p-2+m}}\frac{ C^{m-1}m (N-1) k \coth(k) }{a(m-1)}(\tau+t)^{\alpha m-\beta} \\
\leq & \, C^{\frac{(p-1)(m-1)}{p-2+m}}\left(\frac{C^{m-1}m}{2 a^2(m-1)^2}\right)^{\frac{p-1}{p-2+m}} (\tau+t)^{\frac{\alpha p(m-1)+ (\alpha m - 2\beta )(p-1)}{p-2+m}}  \qquad \forall t > 0 \, .
\end{aligned}
\end{equation}
Observe that
\[\alpha m - \beta < \alpha p \frac{m-1}{p-2+m} + (\alpha m - 2 \beta)\frac{p-1}{p-2+m} \]
holds if and only if
\[(p-m)[\alpha(m-1)-\beta]> 0 \, , \]
which is valid thanks to \eqref{eq117} and \eqref{eq122b}


It is therefore apparent that one can choose $ \tau(a, C, p, m, \alpha, \beta, N, k) \ge 1 $ sufficiently large to make \eqref{eq124} (and so \eqref{eq29}) hold. Note that here the extrema of the inequality ($ p=m $ or $ \beta = \alpha(m-1) $) have to be excluded, otherwise one does not have enough degrees of freedom on $C,a$ to make \eqref{eq124} hold).

Moreover, from \eqref{eq121} we deduce that  \eqref{eq30} is fulfilled whenever
\begin{equation}\label{eq131}
\frac{2 m k \coth(k)  (N-1)}{p-2+m} \leq a \, C^{p-m} \, (\tau+t)^{\alpha(p-m)+\beta} \qquad \forall t \ge 0 \, .
\end{equation}
From \eqref{eq122b} it follows that $\alpha(p-m)+\beta>0$. So, \eqref{eq131} is satisfied, providing again that $\tau(a, C, p, m, \alpha, \beta, N, k) \ge 1 $ is sufficiently large. Furthermore, thanks to the last inequality in \eqref{eq100b} (recall that $ \sigma_0 $ and $ \delta_0 $ are defined by \eqref{eq28b} and \eqref{eq28c}, respectively), we have that
\[
\sigma_0(t) -\delta_0(t) \leq \frac{ 2C^{m-1}m (N-1) C_0}{a(m-1)} \, (\tau+t)^{\alpha m - \beta}\,.
\]
Hence, from \eqref{eq122b} we infer that inequality \eqref{eq41} is satisfied provided
\begin{equation}\label{eq125}
\frac{2^{\frac{p-2+m}{m-1}+1}m (N-1) C_0}{(m-1)} \leq a \, C^{p-m} \, (\tau+t)^{\alpha(p-m)+\beta} \qquad \forall t \ge 1 \, .
\end{equation}
Clearly, \eqref{eq125} holds as long as $\tau(a, C, p, m, \alpha, \beta, N, C_0) \ge 1 $ is sufficiently large. Since \eqref{eq29}, \eqref{eq30} and \eqref{eq41} hold, from Proposition \ref{subsol} we get that $u$ is a subsolution of equation \eqref{eq4}. Finally, \eqref{eq126} yields \eqref{eq108}, so that $u$ is also a subsolution of problem \eqref{eq1}.

\smallskip

\end{proof}

\bf Proof of Theorem \ref{sub p<m}\rm. We use comparison with the barriers constructed in Lemma \ref{subsol3}, see Proposition \ref{thm:cp2}. This yields immediately the claim. \hfill $\qed$

\medskip

\textbf{Acknowledgements.} G.G.~was partially supported by the PRIN Project {``Equazioni alle derivate parziali di tipo ellittico e parabolico: aspetti geometrici, disuguaglianze collegate, e applicazioni''} (Italy). M.M. and F.P.~were partially supported by the GNAMPA Project ``Equazioni diffusive non-lineari in contesti non-Euclidei e disuguaglianze funzionali associate'' (Italy). All authors~have also been supported by the Gruppo Nazionale per l'Analisi Matematica, la Probabilit\`a e le loro Applicazioni (GNAMPA) of the Istituto Nazionale di Alta Matematica (INdAM, Italy).


\begin{thebibliography}{999}

\bibitem{ACP} D. Aronson, M.G. Crandall, L.A. Peletier, \emph{Stabilization of solutions of a degenerate nonlinear diffusion problem}, Nonlinear Anal. \bf 6 \rm (1982), 1001--1022.

\bibitem{BPT} C. Bandle, M.A. Pozio, A. Tesei, \em The Fujita exponent for the Cauchy problem in the hyperbolic space\rm,
J. Differential Equations \bf 251 \rm (2011), 2143--2163.

\bibitem{BGGV} M. Bonforte, F. Gazzola, G. Grillo, J.L. V\'azquez \em Classification of radial solutions to the Emden-Fowler equation on the hyperbolic space\rm,  Calc. Var. Partial Differential Equations \bf 46 \rm (2013), 375--401.

\bibitem{CFG} X. Chen, M. Fila, J.S. Guo, \emph{Boundedness of global solutions of a supercritical parabolic equation}, Nonlinear Anal. {\bf 68} (2008), 621--628.

\bibitem{D} E.B. Davies, ``Heat Kernels and Spectral Theory'', Cambridge Tracts in Mathematics, 92. Cambridge University Press, Cambridge, 1989.

\bibitem{DL}  K. Deng, H.A. Levine, \em The role of critical exponents in blow-up theorems: the sequel\rm, J. Math. Anal. Appl. \bf 243 \rm (2000), 85--126.

\bibitem{FI}  Y. Fujishima, K. Ishige, \em Blow-up set for type I blowing up solutions for a semilinear heat equation\rm, Ann. Inst. H. Poincaré Anal. Non Lin\'eaire \bf 31 \rm (2014), 231--247.

\bibitem{F} H. Fujita, \em On the blowing up of solutions of the Cauchy problem for $u_t=\Delta u+u^{1+\alpha}$\rm, J. Fac. Sci. Univ. Tokyo Sect. I \textbf{13} (1966), 109--124.

\bibitem{GV} V.A. Galaktionov, J.L. V\'azquez, \em Continuation of blowup solutions of nonlinear heat equations in several dimensions\rm, Comm. Pure Appl. Math. \bf 50 \rm (1997), 1--67.

\bibitem{GW} R.E. Greene, H. Wu, ``Function Theory on Manifolds Which Possess a Pole'', Lecture Notes in Mathematics, 699. Springer, Berlin, 1979.

\bibitem{Grig} A. Grigor'yan, \emph{Analytic and geometric background of recurrence and non-explosion of the Brownian motion on Riemannian manifolds}, Bull.~Amer.~Math.~Soc. {\bf 36} (1999), 135--249.

\bibitem{Grig3} A. Grigor'yan, ``Heat Kernel and Analysis on Manifolds'', AMS/IP Studies in Advanced Mathematics, 47, American Mathematical Society, Providence, RI; International Press, Boston, MA, 2009.

\bibitem{GMhyp} G. Grillo, M. Muratori, \emph{Radial fast diffusion on the hyperbolic space}, Proc.~Lond.~Math.~Soc. {\bf 109} (2014), 283--317.

\bibitem{GM2} G. Grillo, M. Muratori, \emph{Smoothing effects for the porous medium equation on Cartan-Hadamard manifolds}, Nonlinear Anal. {\bf 131} (2016), 346--362.

\bibitem{GMPbd} G. Grillo, M. Muratori, F. Punzo, {\em The porous medium equation with large initial data on negatively curved Riemannian manifolds}, to appear on J. Math. Pures Appl., \url{https://doi.org/10.1016/j.matpur.2017.07.021}.

\bibitem{GMPrm} G. Grillo, M. Muratori, F. Punzo, \emph{The porous medium equation with measure data on negatively curved Riemannian manifolds}, accepted for publication on J. Eur. Math. Soc. (JEMS), preprint arXiv: \url{http://arxiv.org/pdf/1507.08883}.

\bibitem{GMV} G. Grillo, M. Muratori, J.L. V\'azquez, \emph{The porous medium equation on Riemannian manifolds with negative curvature. The large-time behaviour}, Adv. Math. 314 (2017), 328--377.

\bibitem{H} K. Hayakawa, \em On nonexistence of global solutions of some semilinear parabolic differential equations\rm, Proc. Japan Acad. \textbf{49} (1973), 503--505.

\bibitem{I} K. Ishige, \em An intrinsic metric approach to uniqueness of the positive Dirichlet problem for parabolic equations in cylinders\rm, J. Differential Equations \bf 158 \rm(1999), 251--290.

\bibitem{I2} K. Ishige, \em An intrinsic metric approach to uniqueness of the positive Cauchy-Neumann problem for parabolic equations\rm, J. Math. Anal. Appl. \bf 276 \rm (2002), 763--790.

\bibitem{IM} K. Ishige, M. Murata, \em Uniqueness of nonnegative solutions of the Cauchy problem for parabolic equations on manifolds or domains\rm, Ann. Scuola Norm. Sup. Pisa Cl. Sci. {\bf 30} (2001), 171--223.

\bibitem{L} H.A. Levine, \em The role of critical exponents in blow-up theorems\rm, SIAM Rev. \bf 32 \rm (1990), 262--288.


\bibitem{MS} G. Mancini, K. Sandeep, \em On a semilinear elliptic equation in $\mathbb{H}^n$\rm, Ann. Scuola Norm. Sup. Pisa Cl. Sci. \bf 7 \rm(2008), 635--671.

\bibitem{MMP} P. Mastrolia, D.D. Monticelli, F. Punzo, \emph{Nonexistence of solutions to parabolic differential inequalities with a potential on Riemannian manifolds}, Math. Ann. \textbf{367} (2017), 929--963.

\bibitem{M} H.P. McKean, \em An upper bound to the spectrum of $\Delta$ on a manifold of negative curvature\rm, J. Differential
Geometry, \bf 4 \rm (1970), 359--366.

\bibitem{MP} E.L. Mitidieri, S.I. Pohozaev, \emph{A priori estimates and the absence of solutions of nonlinear partial differential equations and inequalities}, Tr. Mat. Inst. Steklova \textbf{234} (2001), 1--384; translation in Proc. Steklov Inst. Math. \textbf{234} (2001), 1--362.

\bibitem{MP2} E.L. Mitidieri, S.I. Pohozaev, \emph{Towards a unified approach to nonexistence of solutions for a class of differential inequalities}, Milan J. Math. {\bf 72} (2004), 129--162.

\bibitem{MQV} N. Mizoguchi, F. Quir\'os, J.L. V\'azquez, \em Multiple blow-up for a porous medium equation with reaction\rm,
Math. Ann. \bf 350 \rm (2011), 801--827.



\bibitem{PZ1} L.A. Peletier, J. Zhao, \emph{Source-type solutions of the porous media equation with absorption: the fast diffusion case}, Nonlinear Anal. \textbf{14} (1990), 107--121.

\bibitem{PZ2} L.A. Peletier, J. Zhao, \emph{Large time behaviour of solutions of the porous media equation with absorption: the fast diffusion case}, Nonlinear Anal. \textbf{17} (1991), 991--1009.

\bibitem{PT} S.I. Pohozaev, A. Tesei, \emph{Blow-up of nonnegative solutions to quasilinear parabolic inequalities}, Atti Accad. Naz. Lincei Cl. Sci. Fis. Mat. Natur. Rend. Lincei Mat. Appl. \textbf{11} (2000), 99--109.



\bibitem{Pu3} F. Punzo, {\it Blow-up of solutions to semilinear parabolic equations on Riemannian manifolds with negative sectional curvature}, J. Math. Anal. Appl. {\bf 387} (2012), 815--827.

\bibitem{Pu1} F.  Punzo, {\it Support  properties  of  solutions  to nonlinear  parabolic  equations  with variable density in the hyperbolic space}, Discrete Contin. Dyn. Syst. Ser. S \textbf{5} (2012), 657--670.

\bibitem{PuAA} F. Punzo, {\it Uniqueness and non-uniqueness of solutions to quasilinear parabolic equations with a singular coefficient on weighted Riemannian manifolds}, Asymptot. Anal. {\bf 79} (2012), 273--301.

\bibitem{Pu2} F. Punzo, {\it Well-posedness of the Cauchy problem for nonlinear parabolic equations with  variable  density  in  the  hyperbolic  space}, NoDEA Nonlinear Differential Equations Appl. \textbf{19} (2012), 485--501.


\bibitem{Q} P. Quittner, {\it The decay of global solutions of a semilinear heat equation}, Discrete Contin. Dyn. Syst. {\bf 21} (2008), 307--318.

\bibitem{Sacks} P.A. Sacks, {\it Global beahvior for a class of nonlinear evolution equations}, SIAM J. Math. Anal. {\bf 16} (1985), 233--250.

\bibitem{SGKM}  A.A. Samarskii, V.A. Galaktionov, S.P. Kurdyumov, A.P. Mikhailov, ``Blow-up in Quasilinear Parabolic Equations'', De Gruyter Expositions in Mathematics, 19. Walter de Gruyter \& Co., Berlin, 1995.

\bibitem{S} P. Souplet, \emph{Morrey spaces and classification of global solutions for a supercritical semilinear heat equation in
$\mathbb R^n$}, J. Funct. Anal. {\bf 272} (2017), 2005--2037.

\bibitem{Vaz1} J.L. V\'azquez, {\it The problems of blow-up for nonlinear heat equations. Complete blow-up and avalanche formation},  Atti Accad. Naz. Lincei Cl. Sci. Fis. Mat. Natur. Rend. Lincei Mat. Appl. \textbf{15} (2004), 281--300.

\bibitem{Vaz07} J.L. V\'azquez, ``The Porous Medium Equation. Mathematical Theory'', Oxford Mathematical Monographs. The Clarendon Press, Oxford University Press, Oxford, 2007.

\bibitem{VazH} J.L. V\'azquez, {\it Fundamental solution and long time behavior of the porous medium equation in hyperbolic space}, J. Math. Pures Appl. {\bf 104} (2015), 454--484.

\bibitem{WY} Z. Wang, J. Yin, {\it A note on semilinear heat equation in hyperbolic space}, J. Differential Equations {\bf 256} (2014), 1151--1156.

\bibitem{WY2} Z. Wang, J. Yin, {\it Asymptotic behaviour of the lifespan of solutions for a semilinear heat equation in hyperbolic space}, Proc. Roy. Soc. Edinburgh Sect. A {\bf 146} (2016) 1091--1114.

\bibitem{W} F.B. Weissler, \em $L^p$-energy and blow-up for a semilinear heat equation\rm, Proc. Sympos. Pure Math. \bf 45 \rm (1986), 545--551.

\bibitem{Xin} Y.L. Xin, ``Geometry of Harmonic Maps'', Progress in Nonlinear Differential Equations and their Applications, 23. Birkh\"{a}user Boston, Inc., Boston, MA, 1996.

\bibitem{Y} E. Yanagida, \emph{Behavior of global solutions of the Fujita equation}, Sugaku Expositions {\bf 26} (2013), 129--147.

\bibitem{Z} Q.S. Zhang, \em Blow-up results for nonlinear parabolic equations on manifolds\rm, Duke Math. J. \bf 97 \rm (1999), 515--539.

\end{thebibliography}
\end{document}